\documentclass{amsart}

\usepackage{amsmath,amssymb,amsthm,verbatim}
\usepackage{mathtools}
\usepackage{hyperref}
\usepackage{graphicx}
\usepackage{xcolor}
\usepackage{mathrsfs}
\usepackage[shortalphabetic]{amsrefs}
\usepackage[margin=3.5cm]{geometry}

\hypersetup{colorlinks=false}

\newtheorem{thm}{Theorem}[section]
\newtheorem{theorem}[thm]{Theorem}
\newtheorem{prop}[thm]{Proposition}
\newtheorem{proposition}[thm]{Proposition}
\newtheorem{lem}[thm]{Lemma}
\newtheorem{lemma}[thm]{Lemma}
\newtheorem{cor}[thm]{Corollary}
\newtheorem{corollary}[thm]{Corollary}
\newtheorem{defn}[thm]{Definition}
\newtheorem{definition}[thm]{Definition}
\newtheorem{rem}[thm]{Remark}
\newtheorem{remark}[thm]{Remark}
\newtheorem{claim}[thm]{Claim}

\newcommand{\Kosc}{K_{\mathrm{osc}}}
\newcommand{\R}{\mathbb{R}}
\newcommand{\Sph}{\mathbb{S}}
\renewcommand{\S}{\Sph}

\newcommand{\N}{\mathbb{N}}
\newcommand{\Z}{\mathbb{Z}}
\newcommand{\C}{\mathbb{C}}
\newcommand{\eps}{\varepsilon}
\newcommand{\dd}{\mathrm{d}}
\newcommand{\p}{\partial}
\newcommand{\ip}[2]{\langle #1,#2\rangle}

\newcommand{\Ds}{\p_s}

\newcommand{\Rot}{\mathsf{R}}
\newcommand{\K}{\mathcal{K}}

\allowdisplaybreaks[2]

\title[$L^2(\dd s)$-gradient flow for $E_m$]{On the generalised ideal flow of closed planar curves}
\author[J. McCoy, G. Wheeler]{James McCoy\and Glen Wheeler}

\thanks{Financial support from Discovery Projects DP180100431, DP250103952, DP250101080 and Future Fellowship FT250100880 of the Australian Research Council is gratefully acknowledged.  The first author also thanks the Chinese Academy of Sciences President's International Fellowship Initiative Visiting Fellowship scheme, whose grant 2024PVA0042 partially supported this research.  Notwithstanding several subsequent years to complete, the authors would like to dedicate this work to Prof Graham Williams on the occasion of his appointment as an Emeritus Professor at the University of Wollongong.  They each thank him for his many years of mentorship.}
\address{Department of Mathematical and Geospatial Sciences, School of Science, Royal Melbourne Institute of Technology, 124 La Trobe Street, Melbourne, Victoria 3000, Australia}
\email{james.mccoy@rmit.edu.au}
\address{Institute for Mathematics and its Applications, School of Mathematics and Applied Statistics, University of Wollongong, Northfields Ave, Wollongong, NSW 2500, Australia}
\email{glenw@uow.edu.au}


\begin{document}

\begin{abstract}
For each integer $m\ge0$ we study the \emph{$m$-ideal energy}
\[
E_m[\gamma]:=\frac12\int_\gamma k_{s^m}^2\,ds
\]
on closed immersed planar curves, where $k$ is signed curvature and $s$ is arclength; $k^2_{s^m} := (k_{s^m})^2$.
The $m$-ideal energies contain Euler's elastic energy and the Dirichlet energy for the curvature scalar as special cases ($m=0,1$).

We completely classify the closed smooth critical points of $E_m$ for all $m\ge1$: they are precisely the round multiply-covered circles.
For the steepest descent $L^2(ds)$-gradient flow of $E_m$, the \emph{$m$-ideal flow}, we prove that for each nonzero turning number there is a curvature-oscillation threshold such that every canonical relaxed flow starting from $W^{2,2}$ initial data below this threshold is immortal and exponentially asymptotic in the smooth topology to a round multiply-covered circle.
We also prove that every immortal canonical relaxed trajectory with bounded unnormalised length converges to the corresponding circle.

We furthermore treat rough initial data of class $W^{2,2}$; such data typically has infinite $E_m$ energy when $m\ge1$.
In the small-curvature-oscillation basin, every such curve generates a unique canonical relaxed length-normalised flow, smooth for every positive time, continuously dependent on the initial data, and smoothly convergent to the multiply-covered circle.
These results are known in the $m=0$ case, substantially strengthen existing work in the $m=1$ case, and are new for $m>1$.
\end{abstract}
\maketitle
\tableofcontents
\section{Introduction}\label{S:intro}
Let $\gamma:\S\to\R^2$ be a smooth closed immersed curve.
We study the family of \emph{$m$-ideal energies $E_m$}:
\begin{equation}\label{E:Em}
E_m[\gamma]:=\frac12\int_{\gamma} k_{s^m}^2\,\dd s .
\end{equation}
The case $m=0$ is Euler's elastic energy, and the case $m=1$ is the ideal energy of closed planar curves, or the Dirichlet energy of curvature.
In this article we study both smooth closed immersed critical points of $E_m$, called \emph{$m$-ideal curves}, and the steepest descent $L^2(ds)$-gradient flow of $E_m$, called the \emph{$m$-ideal flow}.

\subsection{On \texorpdfstring{$m$}{m}-ideal curves}
Euler's elastic energy $E_0$ is well-known and well-studied.
Historically $E_1$ arose in the Bernoulli-Euler story behind Euler's spiral: in an inverse elasticity problem posed by James Bernoulli, one asks for the natural shape of a thin lamina that becomes straight when a weight is attached at one end. 
Once the strip is straightened, the bending moment at arclength distance $s$ from the force is $M=Fs$; assuming the strip does not stretch, the curvature of the unstressed lamina is therefore proportional to arclength, $k(s)=cs$, so the curve is Euler's spiral \cite{Levien2008EulerSpiralHistory,Levien2009FromSpiralToSpline}. 
In modern language, Euler spirals are precisely curves with constant $k_s$, so this classical mechanical thought experiment points naturally toward energies that penalise curvature variation. 
Later, Daniel Bernoulli recast the elastica variationally, asking Euler to minimise the stored elastic energy $\int k^2\,ds$ for elastic laminae \cite{Levien2009FromSpiralToSpline}.

In the closed planar setting, stationary points of $E_1$ have been classified as (multiply covered) circles
\cite{AMWW20}.
The Euler-Lagrange operator $\K_m$ of $E_m$ is
\[
\K_m
= (-1)^{m+1}k_{s^{2m+2}}-\frac12 k\,k_{s^m}^2
+k\sum_{r=1}^{m}(-1)^{r+1}\,k_{s^{m-r}}\,k_{s^{m+r}}.
\]
This operator is quasilinear of order $2m+4$ in the immersion, and it would be natural to expect the space of critical curves to grow with $m$.
Here, to the contrary, we prove that the rigidity phenomenon for smooth, closed critical curves extends to all $m$:

\begin{thm}[Stationary solutions]\label{T:stationary-classification}
For every $m\in \mathbb{N}$, a smooth closed immersed planar curve is critical for $E_m$ if and only if it is a round multiply-covered circle.
\end{thm}

\begin{rem}
There are no closed critical points with turning number zero for any $m$, that is, no critical curves regularly homotopic to the lemniscate.
There are no closed critical points at all when $m=0$.
\end{rem}

\begin{rem}
The rigidity phenomenon is sensitive to length constraint: the length-penalised ideal energy $E_1 + \lambda L$ has a wide variety of critical points.
The family of multiply-covered round circles bifurcates into many families, including at least closed lemniscate-like critical points and so-called ``generalised lemniscates''; see \cite{OW261}.
\end{rem}

For our second main result, we present the following gradient inequality, which may have application to related problems and thus be of independent interest.

\begin{thm}[Gradient inequality]\label{T:GI}
Fix $m\ge1$ and $\omega\in\Z\setminus\{0\}$.
There exists $C_{m,\omega}>0$ such that every smooth closed immersed curve $\gamma$ with turning number $\omega[\gamma]=\omega$ satisfies
\begin{equation}\label{E:grad-ineq}
\int_{\gamma}\K_m^2\,\dd s
\ge
C_{m,\omega}\,L[\gamma]^{-(2m+4)}\,E_m[\gamma].
\end{equation}
\end{thm}

\subsection{On \texorpdfstring{$m$}{m}-ideal flows}

The steepest descent $L^2(ds)$-gradient flow for $E_m$ is
\begin{equation*}\label{E:GIF}
\tag{GIF$_m$}
\p_t^\perp\gamma=\K_m\,N,
\end{equation*}
where
\(
\p_t^\perp\gamma:=\ip{\p_t\gamma}{N}N
\)
denotes the normal velocity.
The scale-invariant oscillation of curvature is defined by
\begin{equation}\label{E:Kosc}
\Kosc[\gamma]
:=
L[\gamma]\int_\gamma (k-\bar k)^2\,\dd s ,
\end{equation}
where $\bar k$ is the average of the curvature (see Section \ref{S:firstvar}).
If $\Kosc[\gamma]=0$ then a smooth, closed curve $\gamma$ is a multiply-covered circle.

The flow \eqref{E:GIF} with $m=0$ is studied in \cite{AW26,MW25}.  In particular, there convergence of the normalised flow is established with an initial smallness condition on $\Kosc$.
Similarly, the $m=1$ case is treated in \cite{AMWW20}. There, however, the initial smallness is at the level of the normalised energy (a condition on $L^3E_1$).
Recently, it has been discovered that such a condition is necessary: infinitely many geometrically distinct self-similar solutions for \eqref{E:GIF} with $m=0$ and $m=1$ have been discovered in \cite{AW26}, each of which are non-circular fixed points for the normalised flows.
For further mathematical developments we refer to
\cite{MWW1,Wu2021ShortTimeExistence,MWW2,OW262}.

Our work here extends known results in two fundamental ways.
First, we establish local well-posedness for the flow with rough $W^{2,2}$ data.
Initial data in this class does not typically have bounded energy.
Second, we establish convergence for any $m$ with only a condition on $\Kosc$, unifying the stability argument and improving the $m=1$ result.



\begin{thm}[Local well-posedness for rough data]\label{T:W22-LWP-main}
Fix $m\in \mathbb{N}$ and $\omega\in\mathbb Z\setminus\{0\}$.
Let $\eta_0:\mathbb R/\mathbb Z\to\mathbb R^2$ be a unit-length, centred,
arclength-parametrised $W^{2,2}$ immersed closed curve with turning number
$\omega[\eta_0]=\omega$. Then there exists a time
$T=T(m,\eta_0)>0$ and a unique canonical gauge-fixed length-normalised
$m$-ideal flow
$\eta:\mathbb R/\mathbb Z\times[0,T]\to\mathbb R^2$
 with initial value $\eta\left( \cdot, 0 \right) =\eta_0$, where
$\eta\in C([0,T];W^{2,2})\cap C^\infty((0,T]\times\mathbb R/\mathbb Z)$.
For every positive time in the interval of existence, $\eta$ is a classical solution. The solution depends
continuously on the initial datum in the $W^{2,2}$ topology, locally around
$\eta_0$, and is the limit of every smooth approximating sequence preserving
closedness and turning number.
\end{thm}

The basin of attraction result is as follows.

\begin{thm}[Small-$\Kosc$ basin of attraction]\label{T:small-Kosc}
Fix $m\in \mathbb{N}$ and $\omega\in\Z\setminus\{0\}$.  There exists
$\varepsilon^{W^{2,2}}_{m,\omega}>0$ with the following property.  Let
$\eta_0:\mathbb R/\mathbb Z\to\mathbb R^2$ be a unit-length, centred, arclength-parametrised $W^{2,2}$ immersed closed curve with
$\omega[\eta_0]=\omega$ and
\[
\Kosc[\eta_0]<\varepsilon^{W^{2,2}}_{m,\omega}.
\]
Then there exists a unique canonical gauge-fixed length-normalised $m$-ideal flow 
$\eta:\mathbb R/\mathbb Z\times[0,\infty)\to\mathbb R^2$
with initial value $\eta_0$, satisfying
$\eta\in C([0,\infty);W^{2,2})\cap C^\infty((0,\infty)\times\mathbb R/\mathbb Z).$
For every positive time it is a classical solution.  The flow depends continuously on the initial datum in the
$W^{2,2}$ topology inside the small-$\Kosc$ basin and is the limit, on every compact time interval, of every
smooth approximating sequence preserving closedness and turning number.  Moreover
\begin{equation}\label{eq:intro-small-Kosc-decay}
\Kosc[\eta(\tau)]
\le
C_{m,\omega}\Kosc[\eta_0]e^{-c_{m,\omega}\tau}
\end{equation}
and $\eta(\cdot,\tau)\to\eta_\infty$ in $C^\infty$ as $
\qquad(\tau\to\infty),$ where $\eta_\infty$ is the unit-length $\omega$-circle in the barycentric gauge.
Equivalently, after undoing the length normalisation for any closed $W^{2,2}$ initial curve whose unit-length
centred representative satisfies the same hypotheses, the unique canonical relaxed unnormalised flow converges
modulo translations and reparametrisation to an $\omega$-circle of finite positive length.
\end{thm}

We do not expect bounded length to be a generic feature of the flow (as the existence of expanders shows \cite{AW26}).
Clearly trajectories that are asymptotic to circles have bounded length.
Our final result proves that all immortal canonical relaxed trajectories with bounded unnormalised length are asymptotically circular.

\begin{thm}[Bounded length implies convergence]\label{T:bounded-length-convergence}
Fix $m\in \mathbb{N}$ and $\omega\in\Z\setminus\{0\}$.  Let
$\gamma_0:\mathbb R/\mathbb Z\to\mathbb R^2$ be a closed $W^{2,2}$ immersed curve with
$\omega[\gamma_0]=\omega$, and let $\gamma$ be an immortal canonical relaxed unnormalised $m$-ideal flow
with initial value $\gamma_0$.  Then its length-normalised barycentred representative is, on every compact time interval, the unique canonical
relaxed solution obtained from the local theory of Theorem~\ref{T:W22-LWP-main}, and $\gamma$ is smooth for every
positive time.
If
\[
\sup_{t\ge0}L[\gamma(\cdot,t)]<\infty,
\]
then $\gamma$ converges, modulo translations and reparametrisation, to an $\omega$-circle of finite positive
length.  Equivalently, the length-normalised barycentred representative converges in $C^\infty$ to the unit-length
$\omega$-circle.  The convergence is exponential in the length-normalised time, and hence exponential in the
original time after a finite positive time shift.  The canonical relaxed flow is unique and is the positive-time
limit of every smooth approximating sequence preserving closedness and turning number.
\end{thm}


\subsection{Organisation}

Section~\ref{S:firstvar} computes the first variation of $E_m$ and records the basic dissipation identities.
Section~\ref{S:rigid} proves the sum-of-squares identity behind the rigidity theorem and proves Theorem~\ref{T:stationary-classification}.
Section~\ref{S:grad-ineq} proves the gradient inequality, Theorem~\ref{T:GI}.
Section~\ref{S:existence} records the fixed-gauge local theory, the barycentred length-normalised gauge, and the
curvature-coordinate formulation used for rough initial data; in particular it proves Theorem~\ref{T:W22-LWP-main}.
Section~\ref{S:len-normalised} develops the length-normalised barycentred flow, records the intrinsic evolution
identities and a priori Sobolev estimates, proves the small-$\Kosc$ basin of attraction for rough data,
Theorem~\ref{T:small-Kosc}, and proves the bounded-length convergence theorem,
Theorem~\ref{T:bounded-length-convergence}.  The perturbative remainder estimates needed for the small-$\Kosc$
basin are proved in the body of Section~\ref{S:len-normalised}.
Appendices~\ref{A:normalised-evolution} and~\ref{A:normalised-apriori} contain the remaining detailed proofs of the
evolution identities and Sobolev estimates used in Section~\ref{S:len-normalised}.  Appendix~\ref{SS:linear-circle}
records the normal-graph linearisation about an $\omega$-circle.


\section{Geometric preliminaries and first variation}\label{S:firstvar}

Throughout, $\gamma:\Sph\to\R^2$ is a smooth closed immersion.
Write $s$ for arclength, $T:=\Ds\gamma$ for the unit tangent, and choose unit normal
\(
N:=\Rot T 
\),
where $\Rot(x,y):=(-y,x)$ so that $\{T,N\}$ is positively oriented.
Let $k$ denote the signed curvature; then the Frenet system is
\(
\Ds T=kN\), \(\Ds N=-kT
\).
We write the turning number as
\begin{equation}\label{E:turning}
\omega[\gamma]:=\frac{1}{2\pi}\int_\gamma k\,\dd s\in\Z.
\end{equation}
The turning number is preserved by any differentiable evolution.
The average of the curvature is 
\[
\bar k:=\frac{1}{L[\gamma]}\int_\gamma k\,\dd s
=\frac{2\pi\omega}{L[\gamma]} \mbox{.} \]

We consider normal variations $\gamma_\varepsilon$ of $\gamma$, that is, $\gamma_\eps=\gamma+\eps\varphi N$ with $\varphi\in C^{\infty}(\Sph)$.

\begin{lem}[First variation]\label{L:firstvar}
Let $m\in\N$.
Then
\begin{equation}\label{E:firstvar}
\left.\frac{\dd}{\dd\eps} E_m[\gamma_\eps] \right|_{\eps=0}= -\int_{\gamma} \K_m\,\varphi\,\dd s.
\end{equation}
where $\K_m$ is the scalar differential operator $\K_m=\K_m[k]$ of order $2m+2$ given by
\begin{equation}\label{E:Km}
\K_m
= (-1)^{m+1}k_{s^{2m+2}} - \frac12 k\,k_{s^m}^2
+ k\sum_{j=1}^{m}(-1)^{j+1} k_{s^{m-j}}\,k_{s^{m+j}}.
\end{equation}
\end{lem}
\begin{proof}
Fix a smooth normal variation $\{\gamma_\eps\}_{|\eps|<\eps_0}$ of $\gamma$, with
\[
\left.\p_\eps\gamma_\eps\right|_{\eps=0}=\varphi\,N.
\]
All $s$-derivatives below are taken with respect to the arclength parameter of $\gamma$ at $\eps=0$.

A standard computation gives
\begin{equation}\label{eq:ds-var}
\left.\p_\eps ds_\eps\right|_{\eps=0}=-(k\varphi)\,ds.
\end{equation}
Moreover the commutator between $\p_\eps$ and $\p_s$ is
\begin{equation}\label{eq:commutator}
\left.\p_\eps\p_s\right|_{\eps=0}=\p_s\left.\p_\eps\right|_{\eps=0}+(k\varphi)\p_s,
\qquad\text{equivalently}\qquad
\left.\p_\eps(f_s)\right|_{\eps=0}=(\p_\eps f)_s+(k\varphi)f_s.
\end{equation}
Using \eqref{eq:commutator} with $f=\gamma$ and the Frenet system, one finds
\[
\left.\p_\eps T\right|_{\eps=0}
=\left.\p_\eps\gamma_s\right|_{\eps=0}
=(\varphi N)_s+(k\varphi)\gamma_s
=\varphi_s N-\varphi kT+k\varphi T
=\varphi_s N,
\]
hence $\left.\p_\eps N\right|_{\eps=0}=-\varphi_sT$. Differentiating $T_s=kN$ in $\eps$ and taking the $N$-component yields
\begin{equation}\label{eq:k-var}
\left.\p_\eps k\right|_{\eps=0}=\varphi_{ss}+k^2\varphi.
\end{equation}

Write $D:=\p_s$. From \eqref{eq:commutator} we have $\p_\eps D = D\p_\eps + (k\varphi)D$.
Iterating this identity yields the following  binomial-coefficient formula.

\begin{claim}\label{clm:Ds-m-var}
For every smooth scalar function $f$ and every integer $m\ge 1$,
\begin{equation}\label{eq:Ds-m-var}
\left.\p_\eps(D^mf)\right|_{\eps=0}
=D^m\left(\left.\p_\eps f\right|_{\eps=0}\right)
+\sum_{j=0}^{m-1}\binom{m}{j+1}\,D^j(k\varphi)\,D^{m-j}f.
\end{equation}
\end{claim}
\begin{proof}[Proof of Claim~\ref{clm:Ds-m-var}]
For $m=1$ this is exactly \eqref{eq:commutator}.
Assume \eqref{eq:Ds-m-var} holds for some $m\ge 1$. Then, using $\p_\eps(Dg)=D(\p_\eps g)+(k\varphi)Dg$ and the induction hypothesis,
\begin{align*}
\p_\eps(D^{m+1}f)
&=\p_\eps\bigl(D(D^mf)\bigr)
=D\bigl(\p_\eps(D^mf)\bigr)+(k\varphi)D^{m+1}f\\
&=D\Bigl(D^m(\p_\eps f)+\sum_{j=0}^{m-1}\binom{m}{j+1}D^j(k\varphi)\,D^{m-j}f\Bigr)
+(k\varphi)D^{m+1}f\\
&=D^{m+1}(\p_\eps f)
+\sum_{j=0}^{m-1}\binom{m}{j+1}D^{j+1}(k\varphi)\,D^{m-j}f
+\sum_{j=0}^{m-1}\binom{m}{j+1}D^j(k\varphi)\,D^{m+1-j}f
\\&\qquad+\binom{m}{1}D^0(k\varphi)\,D^{m+1}f\\
&=D^{m+1}(\p_\eps f)
+\sum_{j=1}^{m}\binom{m}{j}D^{j}(k\varphi)\,D^{m+1-j}f
+\sum_{j=0}^{m-1}\binom{m}{j+1}D^j(k\varphi)\,D^{m+1-j}f.
\end{align*}
By Pascal's identity $\binom{m}{j}+\binom{m}{j+1}=\binom{m+1}{j+1}$ (for $1\le j\le m-1$), and collecting endpoints,
this becomes precisely \eqref{eq:Ds-m-var} with $m$ replaced by $m+1$.
\end{proof}

Set
\[
u:=k_{s^m}=D^mk.
\]
By definition, $E_m[\gamma_\eps]=\frac12\int_{\gamma_\eps} u_\eps^2\,ds_\eps$, hence
\begin{equation}\label{eq:dEm-start}
\left.\frac{\dd}{\dd\eps} E_m[\gamma_\eps]\right|_{\eps=0}
=\int_\gamma u\,\left.\p_\eps u_\eps\right|_{\eps=0}\,ds
+\frac12\int_\gamma u^2\,\left.\p_\eps ds_\eps\right|_{\eps=0}.
\end{equation}
Using \eqref{eq:ds-var}, the second term is
\begin{equation}\label{eq:length-term}
\frac12\int_\gamma u^2\,\left.\p_\eps ds_\eps\right|_{\eps=0}
=-\frac12\int_\gamma k\,u^2\,\varphi\,ds.
\end{equation}
Next, apply Claim~\ref{clm:Ds-m-var} with $f=k$:
\[
\left.\p_\eps u_\eps\right|_{\eps=0}
=\left.\p_\eps(D^mk)\right|_{\eps=0}
=D^m\left(\left.\p_\eps k\right|_{\eps=0}\right)
+\sum_{j=0}^{m-1}\binom{m}{j+1}\,D^j(k\varphi)\,D^{m-j}k.
\]
Insert \eqref{eq:k-var} into the first term:
\[
D^m\left(\left.\p_\eps k\right|_{\eps=0}\right)=D^m(\varphi_{ss}+k^2\varphi)=\varphi_{s^{m+2}}+D^m(k^2\varphi).
\]
Therefore, splitting the first integral in \eqref{eq:dEm-start} accordingly,
\begin{align}
\int_\gamma u\,\left.\p_\eps u_\eps\right|_{\eps=0}\,ds
&=\int_\gamma u\,\varphi_{s^{m+2}}\,ds
+\int_\gamma u\,D^m(k^2\varphi)\,ds
+\sum_{j=0}^{m-1}\binom{m}{j+1}\int_\gamma u\,D^j(k\varphi)\,D^{m-j}k\,ds\nonumber\\
&=:I_1+I_2+I_3.\label{eq:I1I2I3}
\end{align}

\paragraph{The highest-order term:}
integrating by parts $(m+2)$ times (no boundary terms on a closed curve),
\begin{equation}\label{eq:I1}
I_1=\int_\gamma u\,\varphi_{s^{m+2}}\,ds
=(-1)^{m+2}\int_\gamma u_{s^{m+2}}\,\varphi\,ds
=(-1)^{m}\int_\gamma k_{s^{2m+2}}\,\varphi\,ds.
\end{equation}

\paragraph{The \texorpdfstring{$k^2$}{k2}-term:}
Integrating by parts $m$ times,
\begin{equation}\label{eq:I2}
I_2=\int_\gamma u\,D^m(k^2\varphi)\,ds
=(-1)^m\int_\gamma u_{s^m}\,k^2\,\varphi\,ds
=(-1)^m\int_\gamma k_{s^{2m}}\,k^2\,\varphi\,ds.
\end{equation}

\paragraph{The commutator terms and the  binomial identity:}
for each $j$ in $I_3$, integrate by parts $j$ times to move $D^j$ off $(k\varphi)$:
\begin{equation}\label{eq:I3-IBP}
I_3
=\int_\gamma k\,\varphi\,
\sum_{j=0}^{m-1}(-1)^j\binom{m}{j+1}\,D^j\!\bigl(u\,D^{m-j}k\bigr)\,ds.
\end{equation}
We now simplify the sum in \eqref{eq:I3-IBP}.

\begin{claim}\label{clm:binomial-diff-identity}
For every smooth scalar function $f$ and $m\ge 1$, letting $u:=D^mf$, one has the pointwise identity
\begin{equation}\label{eq:diff-identity}
\sum_{j=0}^{m-1}(-1)^j\binom{m}{j+1}\,D^j\!\bigl(u\,D^{m-j}f\bigr)
=u^2+\sum_{r=1}^{m-1}(-1)^r\,D^{m-r}f\,D^{m+r}f.
\end{equation}
\end{claim}

\begin{proof}[Proof of Claim~\ref{clm:binomial-diff-identity}]
Expand each $D^j(u\,D^{m-j}f)$ by Leibniz:
\[
D^j\!\bigl(u\,D^{m-j}f\bigr)
=\sum_{p=0}^{j}\binom{j}{p}\,D^{p}u\,D^{j-p}(D^{m-j}f)
=\sum_{p=0}^{j}\binom{j}{p}\,D^{m+p}f\,D^{m-p}f.
\]
Hence the left-hand side of \eqref{eq:diff-identity} equals
\[
\sum_{p=0}^{m-1} D^{m+p}f\,D^{m-p}f
\underbrace{\sum_{j=p}^{m-1}(-1)^j\binom{m}{j+1}\binom{j}{p}}_{=:C_{m,p}}.
\]
It remains to show
\begin{equation}\label{eq:Cmp}
C_{m,p}=(-1)^p\qquad\text{for }0\le p\le m-1.
\end{equation}
We use coefficient extraction. Recall the definition 
\[
\binom{j}{p}=[x^p](1+x)^j,
\]
where $[x^p]$ denotes the coefficient of $x^p$ in a power series.

Since $\binom{j}{p}=0$ for $j<p$, we may extend the sum to start at $j=0$:
\[
C_{m,p}=[x^p]\sum_{j=0}^{m-1}(-1)^j\binom{m}{j+1}(1+x)^j.
\]
Set $y:=1+x$ and consider the finite sum
\[
S_m(y):=\sum_{j=0}^{m-1}(-1)^j\binom{m}{j+1}y^j.
\]
Multiply by $y$ and shift index $r=j+1$:
\begin{multline*}
y\, S_m(y)
=\sum_{j=0}^{m-1}(-1)^j\binom{m}{j+1}y^{j+1}
=\sum_{r=1}^{m}(-1)^{r-1}\binom{m}{r}y^{r}
=-\sum_{r=1}^{m}(-1)^{r}\binom{m}{r}y^{r}\\
= -\Bigl(\sum_{r=0}^{m}(-1)^{r}\binom{m}{r}y^{r}-1\Bigr)
=-( (1-y)^m-1)=1-(1-y)^m.
\end{multline*}
Therefore
\[
S_m(y)=\frac{1-(1-y)^m}{y}.
\]
Substituting $y=1+x$ gives
\[
\sum_{j=0}^{m-1}(-1)^j\binom{m}{j+1}(1+x)^j
=\frac{1-(1-(1+x))^m}{1+x}
=\frac{1-(-x)^m}{1+x}.
\]
Consequently,
\[
C_{m,p}=[x^p]\frac{1-(-x)^m}{1+x}.
\]
Now note that for $0\le p\le m-1$, the term $(-x)^m/(1+x)$ contributes only powers $x^q$ with $q\ge m$,
so it does  {not} affect the coefficient of $x^p$. Hence
\[
C_{m,p}=[x^p]\frac{1}{1+x}.
\]
Finally,
\[
\frac{1}{1+x}=\sum_{q=0}^{\infty}(-1)^q x^q
\quad\Longrightarrow\quad
[x^p]\frac{1}{1+x}=(-1)^p,
\]
which proves the claim.
Substituting back yields \eqref{eq:diff-identity}.
\end{proof}

Applying Claim~\ref{clm:binomial-diff-identity} with $f=k$ in \eqref{eq:I3-IBP} gives
\begin{equation}\label{eq:I3-simplified}
I_3
=\int_\gamma k\,\varphi\,
\Bigl(u^2+\sum_{r=1}^{m-1}(-1)^r\,k_{s^{m-r}}\,k_{s^{m+r}}\Bigr)\,ds.
\end{equation}

Combine \eqref{eq:dEm-start}, \eqref{eq:length-term}, \eqref{eq:I1}, \eqref{eq:I2}, \eqref{eq:I3-simplified}.
Noting that \eqref{eq:I2} can be written as
\[
I_2=\int_\gamma k\,\varphi\,\bigl((-1)^m k\,k_{s^{2m}}\bigr)\,ds,
\]
we may fold it into the alternating sum by extending $r$ to $m$:
\[
(-1)^m k\,k_{s^{2m}}=\;(-1)^m k_{s^{m-m}}\,k_{s^{m+m}}.
\]
Hence
\begin{align*}
\left.\frac{\dd}{\dd\eps} E_m[\gamma_\eps]\right|_{\eps=0}
&=(-1)^m\int_\gamma k_{s^{2m+2}}\varphi\,ds
+\int_\gamma k\,\varphi\Bigl(u^2+\sum_{r=1}^{m}(-1)^r\,k_{s^{m-r}}\,k_{s^{m+r}}\Bigr)\,ds
-\frac12\int_\gamma k\,u^2\,\varphi\,ds\\
&=\int_\gamma
\Bigl[(-1)^m k_{s^{2m+2}}
+\frac12 k\,u^2
+k\sum_{r=1}^{m}(-1)^r\,k_{s^{m-r}}\,k_{s^{m+r}}
\Bigr]\varphi\,ds.
\end{align*}
Recalling $u=k_{s^m}$ and rewriting the last sum with $(-1)^r=-(-1)^{r+1}$, this is exactly
\[
\left.\frac{\dd}{\dd\eps}\right|_{\eps=0}E_m[\gamma_\eps]
=-\int_\gamma \K_m\,\varphi\,ds,
\]
with
\[
\K_m
= (-1)^{m+1}k_{s^{2m+2}}-\frac12 k\,k_{s^m}^2
+k\sum_{r=1}^{m}(-1)^{r+1}\,k_{s^{m-r}}\,k_{s^{m+r}}.
\]
This is \eqref{E:firstvar}, \eqref{E:Km}.
\end{proof}

Let us record the energy identity for later use:

\begin{cor}[Gradient flow structure]\label{C:diss}
Let $\gamma(\cdot,t)$ solve \eqref{E:GIF} smoothly.
Then
\begin{align}
\frac{\dd}{\dd t}E_m[\gamma(\cdot,t)] &= -\int_{\gamma} \K_m^2\,\dd s.\label{E:diss-Em}
\end{align}
In particular, $E_m$ is non-increasing and $\int_0^\infty\!\int_{\gamma} \K_m^2\,\dd s\,\dd t\le E_m[\gamma_0]$.
\end{cor}


\section{Rigidity of stationary solutions}\label{S:rigid}

A structural feature of the Euler-Lagrange equation $\K_m\equiv 0$ is a sum-of-squares identity.

\begin{defn}[Auxiliary quantities]\label{D:MNQ}
Define
\[
M_m:=k_{s^{2m+1}},\qquad
N_m:=\frac12(-1)^m k_{s^m}^2+\sum_{j=0}^{m-1}(-1)^j k_{s^j}\,k_{s^{2m-j}},
\]
and
\[
Q_m:=M_m^2+N_m^2\,.
\]
\end{defn}

\begin{lem}[Derivative identity]\label{L:Qs}
For every smooth planar curve,
\begin{equation}\label{E:Qs}
\Ds Q_m = 2(-1)^{m+1}k_{s^{2m+1}}\,\K_m.
\end{equation}
Equivalently, writing $\theta$ for the tangential angle (so that $\theta_s=k$), the complex quantity
\[
\mathcal Q_m := (M_m+iN_m)e^{-i\theta}
\]
satisfies
\begin{equation}\label{E:complexQ}
\Ds\mathcal Q_m = (-1)^{m+1}\K_m\,e^{-i\theta}.
\end{equation}
\end{lem}

\begin{proof}
We prove the complex identity \eqref{E:complexQ}; \eqref{E:Qs} then follows by taking the derivative of $Q_m=|\mathcal Q_m|^2$.

We claim that the following two identities hold pointwise:
\begin{align}
\label{eq:Nm-s}
(N_m)_s &= k\,M_m,\\[0.25em]
\label{eq:Km-real}
(-1)^{m+1}\K_m &= (M_m)_s + k\,N_m.
\end{align}

\smallskip
\noindent {Proof of \eqref{eq:Km-real}.}
Since $(M_m)_s=k_{s^{2m+2}}$, the explicit formula \eqref{E:Km} gives
\[
(-1)^{m+1}\K_m
= k_{s^{2m+2}}+\frac12(-1)^m k\,k_{s^m}^2
+ k\sum_{j=1}^{m}(-1)^{m+j} k_{s^{m-j}}\,k_{s^{m+j}}.
\]
In the sum, set $r:=m-j\in\{0,1,\dots,m-1\}$; then $m+j=2m-r$ and $(-1)^{m+j}=(-1)^{2m-r}=(-1)^r$.
Hence
\[
k\sum_{j=1}^{m}(-1)^{m+j} k_{s^{m-j}}\,k_{s^{m+j}}
= k\sum_{r=0}^{m-1}(-1)^r k_{s^{r}}\,k_{s^{2m-r}}.
\]
Together with the definition of $N_m$, this yields \eqref{eq:Km-real}.

\smallskip
\noindent {Proof of \eqref{eq:Nm-s}.}
Differentiate the definition of $N_m$:
\begin{align*}
(N_m)_s
&= (-1)^m k_{s^m}k_{s^{m+1}}
+ \sum_{j=0}^{m-1}(-1)^j\Bigl(k_{s^{j+1}}k_{s^{2m-j}} + k_{s^j}k_{s^{2m-j+1}}\Bigr)\\
&= (-1)^m k_{s^m}k_{s^{m+1}}
+ \sum_{i=1}^{m}(-1)^{i-1}k_{s^{i}}k_{s^{2m-i+1}}
+ \sum_{j=0}^{m-1}(-1)^j k_{s^j}k_{s^{2m+1-j}},
\end{align*}
where in the middle sum we reindexed $i=j+1$.
Now observe that for each $i\in\{1,\dots,m-1\}$, the term
$k_{s^{i}}k_{s^{2m-i+1}}$ appears in the last two sums with coefficients
$(-1)^{i-1}$ and $(-1)^{i}$, which cancel.
Thus only the last, and first terms respectively, remain:
\[
(N_m)_s
= (-1)^m k_{s^m}k_{s^{m+1}}
+ (-1)^{m-1}k_{s^{m}}k_{s^{m+1}}
+ k\,k_{s^{2m+1}}
= k\,k_{s^{2m+1}}
= k\,M_m,
\]
which is \eqref{eq:Nm-s}.

Using $\theta_s=k$ and the product rule,
\begin{align*}
\Ds\bigl((M_m+iN_m)e^{-i\theta}\bigr)
&=\bigl((M_m)_s+i(N_m)_s - i\theta_s(M_m+iN_m)\bigr)e^{-i\theta}\\
&=\bigl((M_m)_s + kN_m + i\bigl((N_m)_s-kM_m\bigr)\bigr)e^{-i\theta}.
\end{align*}
By \eqref{eq:Nm-s}, the imaginary part vanishes, and by \eqref{eq:Km-real} the real part equals $(-1)^{m+1}\K_m$.
This gives \eqref{E:complexQ}.

Since $Q_m=|\mathcal Q_m|^2$, we have
\[
\Ds Q_m = 2\Re\bigl(\overline{\mathcal Q_m}\,\Ds\mathcal Q_m\bigr)
=2\Re\bigl(\overline{(M_m+iN_m)e^{-i\theta}}\cdot (-1)^{m+1}\K_m e^{-i\theta}\bigr)
=2(-1)^{m+1}M_m\,\K_m,
\]
which is \eqref{E:Qs}.
\end{proof}

\begin{lem}[Mean value of $N_m$]\label{lem:mean-Nm}
For every smooth closed curve,
\begin{equation}\label{eq:mean-Nm}
\int_\gamma N_m\,ds
=\left(m+\frac12\right)(-1)^m\int_\gamma k_{s^m}^2\,ds.
\end{equation}
\end{lem}

\begin{proof}
Using periodicity and integration by parts on $\S$,
for each $j\in\{0,\dots,m-1\}$ we have
\[
\int_\gamma k_{s^j}k_{s^{2m-j}}\,ds
=(-1)^{m-j}\int_\gamma k_{s^m}k_{s^m}\,ds
=(-1)^{m-j}\int_\gamma k_{s^m}^2\,ds.
\]
Multiplying by $(-1)^j$ yields
\[
\int_\gamma (-1)^j k_{s^j}k_{s^{2m-j}}\,ds
=(-1)^m\int_\gamma k_{s^m}^2\,ds,
\]
independent of $j$. Summing over $j=0,\dots,m-1$ and adding the remaining $\frac12(-1)^m k_{s^m}^2$ term
gives \eqref{eq:mean-Nm}.
\end{proof}

\begin{thm}[Rigidity of equilibria]\label{T:rigidity}
Fix $m\ge1$.
A smooth closed immersed curve is a critical point of $E_m$ (equivalently, $\K_m\equiv0$) if and only if it is a
round multiply-covered circle.
\end{thm}

\begin{proof}
Assume $\K_m\equiv 0$. By Lemma~\ref{L:Qs}, \eqref{E:complexQ} reduces to
\[
\Ds\mathcal Q_m = 0,
\]
so $\mathcal Q_m$ is constant on $\S$:
\begin{equation}\label{eq:Qm-constant}
(M_m+iN_m)e^{-i\theta}\equiv C
\qquad\text{for some constant }C\in\C.
\end{equation}
Equivalently,
\begin{equation}\label{eq:MN-as-tangent}
M_m+iN_m = C\,e^{i\theta}.
\end{equation}

Integrating \eqref{eq:MN-as-tangent} over $\gamma$ and using $e^{i\theta}=T$ (in complex notation) yields
\[
\int_\gamma (M_m+iN_m)\,ds = C\int_\gamma e^{i\theta}\,ds
= C\int_\gamma T\,ds
= C\int_\gamma \gamma_s\,ds
= C\bigl(\gamma(L)-\gamma(0)\bigr)
=0,
\]
since $\gamma$ is closed. Taking imaginary parts gives
\begin{equation}\label{eq:Nm-mean-zero}
\int_\gamma N_m\,ds = 0.
\end{equation}
By Lemma~\ref{lem:mean-Nm} we therefore obtain
\[
\left(m+\frac12\right)(-1)^m\int_\gamma k_{s^m}^2\,ds = 0,
\]
hence $k_{s^m}\equiv 0$ on $\S$. Thus $k$ is a polynomial in $s$ of degree at most $m-1$; since $k$ is periodic on $\S$,
it follows that $k$ is constant.

If the constant were zero, then $T_s=0$ and hence $T$ would be constant, which is incompatible with a
nondegenerate closed immersion.  Thus $k$ is a nonzero constant.  A smooth closed planar curve with
constant nonzero curvature is a circle traversed with constant speed; moreover,
\[
\int_\gamma k\,ds = 2\pi\,\omega[\gamma]\in 2\pi\Z
\]
determines the turning number $\omega$. Therefore $\gamma$ is the corresponding multiply-covered circle.  Conversely, if $\gamma$ is a round multiply-covered circle then $k$ is constant, so $k_{s^j}\equiv0$ for every
$j\ge1$.  The explicit formula \eqref{E:Km} gives $\K_m\equiv0$.  Thus every round multiply-covered circle is
critical.
\end{proof}

\begin{proof}[Proof of Theorem~\ref{T:stationary-classification}]
This is Theorem~\ref{T:rigidity}.
\end{proof}


\section{A gradient inequality}\label{S:grad-ineq}

In this section we prove the scale-invariant gradient inequality used in the stability argument.
The proof has two parts.  First we prove a coercive estimate in the small scale-invariant energy regime.
This is done after normalising the length to one, where all constants are dimensionless.  We then scale back.
Second, we prove a weak global gradient inequality.  Combining the small-energy estimate with the weak estimate gives
the global inequality stated in Theorem~\ref{T:GI}.

Throughout this section
\[
\mathcal I_m[\gamma]:=L[\gamma]^{2m+1}E_m[\gamma].
\]
The desired scale-invariant form is
\[
L[\gamma]^{4m+5}\int_\gamma \K_m^2\,ds
\ge C\,\mathcal I_m[\gamma],
\]
or equivalently
\begin{equation}\label{eq:grad-ineq-physical}
\|\K_m\|_{L^2(\gamma)}^2
\ge
C\,L[\gamma]^{-(4m+5)}\mathcal I_m[\gamma]
=
C\,L[\gamma]^{-(2m+4)}E_m[\gamma].
\end{equation}

\subsection{The length-one reduction}

We first work with a smooth closed curve of length $L=1$ and fixed nonzero turning number $\omega\in\Z\setminus\{0\}.$  Then
\[
\bar k=\int_\gamma k\,ds=2\pi\omega
\]
and
\[
\mathcal I_m=E_m=\frac12\int_\gamma k_{s^m}^2\,ds.
\]
It is convenient to set
\[
\kappa:=2\pi\omega,\qquad f:=k-\kappa.
\]
Then
\[
\int_\gamma f\,ds=0,\qquad f_{s^j}=k_{s^j}\quad(j\ge1).
\]

We shall repeatedly use the following consequences of the Poincar\'e--Sobolev--Wirtinger (PSW) inequality                              on a unit-length circle.  If
$u$ has zero average, then
\begin{equation}\label{eq:PW-unit}
\|u\|_{L^2} \le \frac{1}{2\pi}\|u_s\|_{L^2}.
\end{equation}
Moreover, if $u$ has zero average, then
\begin{equation}\label{eq:Sob-unit-zero}
\|u\|_{L^\infty}^2\le 2\|u\|_{L^2}\|u_s\|_{L^2}.
\end{equation}

\begin{lem}[Top-order Poincar\'e lower bound]\label{lem:P-lower-by-Im}
For every smooth closed curve and every $m\ge1$,
\begin{equation}\label{eq:P-lower-by-Im}
\|k_{s^{2m+2}}\|_{L^2(\gamma)}^2
\ge
\left(\frac{2\pi}{L}\right)^{2m+4}\|k_{s^m}\|_{L^2(\gamma)}^2
=
2(2\pi)^{2m+4}L^{-(4m+5)}\mathcal I_m[\gamma].
\end{equation}
In particular, if $L=1$ and
\[
\mathcal P:=\int_\gamma k_{s^{2m+2}}^2\,ds,
\]
then
\begin{equation}\label{eq:P-lower-unit}
\mathcal P
\ge
2(2\pi)^{2m+4}E_m.
\end{equation}
\end{lem}

\begin{proof}
The function $k_{s^m}$ has zero average. Applying \eqref{eq:PW-unit} iteratively on a circle of length $L$ gives
\[
\|k_{s^m}\|_2
\le
\left(\frac{L}{2\pi}\right)^{m+2}\|k_{s^{2m+2}}\|_2.
\]
Squaring and rearranging gives the first inequality.  The final identity follows from
\[
\|k_{s^m}\|_2^2=2E_m=2L^{-(2m+1)}\mathcal I_m.
\]
\end{proof}

\subsection{Splitting of the Euler--Lagrange operator}

From Lemma~\ref{L:firstvar},
\[
\K_m
=
(-1)^{m+1}k_{s^{2m+2}}
-\frac12 k\,k_{s^m}^2
+
k\sum_{j=1}^{m}(-1)^{j+1}k_{s^{m-j}}k_{s^{m+j}}.
\]
Isolating the $j=m$ term gives
\begin{equation}\label{eq:K-splitting}
\K_m
=
(-1)^{m+1}\bigl(k_{s^{2m+2}}+k^2k_{s^{2m}}\bigr)
-\sum_{j=1}^{m-1}(-1)^j k\,k_{s^{m-j}}k_{s^{m+j}}
-\frac12 k\,k_{s^m}^2.
\end{equation}
In the unit-length setting define
\begin{equation}\label{eq:F0-def}
F_0
:=
(-1)^{m+1}\bigl(k_{s^{2m+2}}+\kappa^2k_{s^{2m}}\bigr),
\qquad
R:=\K_m-F_0.
\end{equation}
Thus
\begin{equation}\label{eq:R-def-unit}
R
=
(-1)^{m+1}(k^2-\kappa^2)k_{s^{2m}}
-\sum_{j=1}^{m-1}(-1)^j k\,k_{s^{m-j}}k_{s^{m+j}}
-\frac12 k\,k_{s^m}^2.
\end{equation}

\subsection{Closure control of the resonant Fourier modes}

The only kernel modes of the frozen linear part $F_0$ are the translation modes $p=\pm\omega$.
They are not free Fourier modes of the curvature of a closed curve: closure forces them to be quadratic in the curvature oscillation.

\begin{lem}[Resonant Fourier mode bound]\label{lem:omega-mode-bound}
Assume $L=1$ and write
\[
k(s)=\sum_{p\in\Z}a_p e^{2\pi i p s}.
\]
Then
\begin{equation}\label{eq:omega-mode-bound}
|a_{\omega}|+|a_{-\omega}|
\le
C_{m,\omega}E_m.
\end{equation}
\end{lem}

\begin{proof}
Let $f:=k-\kappa,\qquad \kappa=2\pi\omega.$  Since $\int f\,ds=0$, let $\psi$ be the mean-zero primitive of $f$:
\[
\psi_s=f,\qquad \int_0^1\psi\,ds=0.
\]
After fixing the additive constant in the tangent angle, we may write
\[
\theta(s)=\kappa s+\psi(s),
\qquad
\gamma_s=e^{i\theta(s)}.
\]
Since $\gamma$ is closed,
\[
0=\int_0^1e^{i\theta(s)}\,ds
=
\int_0^1e^{2\pi i\omega s}e^{i\psi(s)}\,ds.
\]
Also $\int_0^1e^{2\pi i\omega s}\,ds=0$, because $\omega\ne0$. Hence
\[
i\int_0^1e^{2\pi i\omega s}\psi(s)\,ds
=
-\int_0^1e^{2\pi i\omega s}\bigl(e^{i\psi(s)}-1-i\psi(s)\bigr)\,ds.
\]
Using
\[
|e^{ix}-1-ix|\le \frac12 x^2
\qquad(x\in\R)
\]
gives
\begin{equation}\label{eq:closure-psi-bound}
\left|\int_0^1e^{2\pi i\omega s}\psi(s)\,ds\right|
\le
\frac12\|\psi\|_2^2.
\end{equation}
Since
\[
f(s)=\sum_{p\ne0}a_p e^{2\pi ips},
\qquad
\psi(s)=\sum_{p\ne0}\frac{a_p}{2\pi i p}e^{2\pi ips},
\]
we have
\[
\int_0^1e^{2\pi i\omega s}\psi(s)\,ds
=
\frac{a_{-\omega}}{2\pi i(-\omega)}.
\]
Therefore
\[
|a_{-\omega}|
\le
\pi|\omega|\,\|\psi\|_2^2.
\]
By the PSW inequality applied $m+1$ times to $\psi$,
\[
\|\psi\|_2^2
\le
(2\pi)^{-2m-2}\|\psi_{s^{m+1}}\|_2^2
=
(2\pi)^{-2m-2}\|k_{s^m}\|_2^2
=
2(2\pi)^{-2m-2}E_m.
\]
Thus $|a_{-\omega}|\le C_{m,\omega}E_m$. Since $k$ is real-valued,
$a_\omega=\overline{a_{-\omega}}$, and the result follows.
\end{proof}

\begin{lem}[Fourier lower bound for $F_0$]\label{lem:F0-lower}
Assume $L=1$ and fix $\omega\in\Z\setminus\{0\}$.
There exist constants $c_\omega>0$ and $C_{m,\omega}>0$ such that
\begin{equation}\label{eq:F0-lower}
\int_\gamma F_0^2\,ds
\ge
c_\omega\,\mathcal P
-
C_{m,\omega}E_m^2,
\qquad
\mathcal P:=\int_\gamma k_{s^{2m+2}}^2\,ds.
\end{equation}
\end{lem}

\begin{proof}
Using the Fourier expansion
\[
k(s)=\sum_{p\in\Z}a_p e^{2\pi ips},
\]
Parseval' identity gives
\[
\mathcal P
=
\sum_{p\in\Z}|a_p|^2(2\pi p)^{4m+4}.
\]
Moreover, by \eqref{eq:F0-def},
\[
F_0(s)
=
\sum_{p\in\Z}
a_p(2\pi)^{2m+2}p^{2m}(p^2-\omega^2)e^{2\pi ips},
\]
and hence
\[
\int_\gamma F_0^2\,ds
=
\sum_{p\in\Z}
|a_p|^2(2\pi)^{4m+4}p^{4m}(p^2-\omega^2)^2.
\]

Set
\[
c_\omega
:=
\min_{p\in\Z\setminus\{0,\pm\omega\}}
\left(1-\frac{\omega^2}{p^2}\right)^2
>0.
\]
For $p\in\Z\setminus\{0,\pm\omega\}$,
\[
p^{4m}(p^2-\omega^2)^2
=
p^{4m+4}\left(1-\frac{\omega^2}{p^2}\right)^2
\ge
c_\omega p^{4m+4}.
\]
Therefore
$$\int_\gamma F_0^2\,ds
\ge
c_\omega
\sum_{p\in\Z\setminus\{0,\pm\omega\}}
|a_p|^2(2\pi p)^{4m+4} 
=
c_\omega\left[
\mathcal P
-
(2\pi\omega)^{4m+4}\bigl(|a_\omega|^2+|a_{-\omega}|^2\bigr)
\right].$$

The resonant contribution is bounded by Lemma~\ref{lem:omega-mode-bound}:
\[
(2\pi\omega)^{4m+4}\bigl(|a_\omega|^2+|a_{-\omega}|^2\bigr)
\le
C_{m,\omega}E_m^2.
\]
This proves \eqref{eq:F0-lower}.
\end{proof}

\subsection{The perturbative remainder}

\begin{lem}[Remainder estimate]\label{lem:R-upper}
Assume $L=1$ and fix $m\ge1$, $\omega\in\Z\setminus\{0\}$.
For every $\delta\in(0,1)$ there exist constants
$\varepsilon_R=\varepsilon_R(m,\omega,\delta)>0,
\qquad
C_{\delta,m,\omega}<\infty,$
such that if
\[
E_m[\gamma]\le \varepsilon_R,
\]
then
\begin{equation}\label{eq:R-upper}
\int_\gamma R^2\,ds
\le
\delta\,\mathcal P
+
C_{\delta,m,\omega}E_m^2.
\end{equation}
\end{lem}

\begin{proof}
By \eqref{eq:R-def-unit},
\[
R
=
(-1)^{m+1}(k^2-\kappa^2)k_{s^{2m}}
-\sum_{j=1}^{m-1}(-1)^j k\,k_{s^{m-j}}k_{s^{m+j}}
-\frac12 k\,k_{s^m}^2.
\]
Hence
\begin{equation}\label{eq:R-sq-split}
\int R^2\,ds
\le
C(m)\left(T_1+\sum_{j=1}^{m-1}T_{2,j}+T_3\right),
\end{equation}
where
\[
T_1:=\int (k^2-\kappa^2)^2k_{s^{2m}}^2\,ds,
\]
\[
T_{2,j}:=\int k^2 k_{s^{m-j}}^2k_{s^{m+j}}^2\,ds
\qquad(1\le j\le m-1),
\]
and
\[
T_3:=\int k^2 k_{s^m}^4\,ds.
\]

We first record the elementary estimates used below.  Since $f=k-\kappa$ has zero average,
\eqref{eq:PW-unit} and \eqref{eq:Sob-unit-zero} give
\begin{equation}\label{eq:f-infty-unit}
\|k-\kappa\|_\infty^2
\le
C_m\|k_{s^m}\|_2^2
=
C_m E_m.
\end{equation}
Similarly, for $1\le r\le m-1$,
\begin{equation}\label{eq:lower-derivative-infty-unit}
\|k_{s^r}\|_\infty^2
\le
C_m\|k_{s^m}\|_2^2
=
C_m E_m.
\end{equation}
Furthermore,
\begin{equation}\label{eq:k-infty-unit}
\|k\|_\infty^2+\|k+\kappa\|_\infty^2
\le
C_\omega(1+E_m).
\end{equation}
Finally, for every $1\le q\le 2m+2$,
\begin{equation}\label{eq:step-up-unit}
\|k_{s^q}\|_2^2
\le
C_{m,q}\mathcal P.
\end{equation}

We estimate the three types of term.

\smallskip
\paragraph{Estimate of \texorpdfstring{$T_1$}{T1}:}
using
\[
k^2-\kappa^2=(k-\kappa)(k+\kappa),
\]
together with \eqref{eq:f-infty-unit}, \eqref{eq:k-infty-unit}, and \eqref{eq:step-up-unit},
\[
T_1
\le
\|k-\kappa\|_\infty^2
\|k+\kappa\|_\infty^2
\|k_{s^{2m}}\|_2^2
\le
C_{m,\omega}E_m(1+E_m)\mathcal P.
\]

\smallskip
\paragraph{Estimate of \texorpdfstring{$T_{2,j}$}{T2j}:}
for $1\le j\le m-1$,
\[
T_{2,j}
\le
\|k\|_\infty^2
\|k_{s^{m-j}}\|_\infty^2
\|k_{s^{m+j}}\|_2^2
\le
C_{m,\omega}E_m(1+E_m)\mathcal P.
\]

\smallskip
\paragraph{Estimate of \texorpdfstring{$T_3$}{T3}:}
using \eqref{eq:k-infty-unit},
\[
T_3
\le
\|k\|_\infty^2\|k_{s^m}\|_\infty^2\|k_{s^m}\|_2^2
\le
C_\omega(1+E_m)\|k_{s^m}\|_\infty^2 E_m.
\]
Since $k_{s^m}$ has zero average,
\[
\|k_{s^m}\|_\infty^2
\le
2\|k_{s^m}\|_2\|k_{s^{m+1}}\|_2
\le
C_m E_m^{1/2}\mathcal P^{1/2}.
\]
Therefore
\[
T_3
\le
C_{m,\omega}(1+E_m)E_m^{3/2}\mathcal P^{1/2}.
\]
By Young's inequality, for every $\delta\in(0,1)$,
\[
T_3
\le
\delta\mathcal P
+
C_{\delta,m,\omega}(1+E_m)^2E_m^3.
\]

Choose $\varepsilon_R=\varepsilon_R(m,\omega,\delta)>0$ so small that
\[
C_{m,\omega}E_m(1+E_m)\le \delta
\]
whenever $E_m\le\varepsilon_R$, and also $\varepsilon_R\le1$.
Then
\[
T_1+\sum_{j=1}^{m-1}T_{2,j}
\le
C(m)\delta\mathcal P,
\]
and
\[
T_3
\le
\delta\mathcal P+C_{\delta,m,\omega}E_m^2,
\]
because $E_m\le1$ implies $E_m^3\le E_m^2$.
Relabeling $\delta$ appropriately in the preceding estimates and using \eqref{eq:R-sq-split} yields
\eqref{eq:R-upper}.
\end{proof}

\subsection{The small-energy gradient inequality}

\begin{thm}[Gradient inequality in the small-energy regime]\label{thm:grad-ineq}
Fix $m\ge1$ and $\omega\in\Z\setminus\{0\}$.
There exist constants
$\varepsilon_{m,\omega}^{\rm sm}>0,
\qquad
C_{m,\omega}^{\rm sm}>0,$
such that every smooth closed immersed curve $\gamma$ with turning number $\omega[\gamma]=\omega$ and $\mathcal I_m[\gamma]<\varepsilon_{m,\omega}^{\rm sm}$
satisfies
\begin{equation}\label{eq:grad-ineq-final}
\int_\gamma \K_m^2\,ds
\ge
C_{m,\omega}^{\rm sm}
L[\gamma]^{-(4m+5)}\mathcal I_m[\gamma]
=
C_{m,\omega}^{\rm sm}
L[\gamma]^{-(2m+4)}E_m[\gamma].
\end{equation}
\end{thm}

\begin{proof}
We first prove the assertion for length-one curves.  Thus $L=1$ and $\mathcal I_m=E_m$.  By the elementary inequality
\[
\|A+B\|_2^2\ge \frac12\|A\|_2^2-\|B\|_2^2,
\]
applied to $\K_m=F_0+R$, Lemmas~\ref{lem:F0-lower} and \ref{lem:R-upper} give, for
$E_m\le\varepsilon_R(m,\omega,c_\omega/8)$,
\[
\int_\gamma \K_m^2\,ds
\ge
\frac12\int_\gamma F_0^2\,ds-\int_\gamma R^2\,ds
\ge
c_1\,\mathcal P-C_1E_m^2,
\]
where $c_1=c_1(m,\omega)>0$ and $C_1=C_1(m,\omega)<\infty$.
By \eqref{eq:P-lower-unit},
\[
\mathcal P\ge c_2(m)E_m.
\]
Therefore
\[
\int_\gamma \K_m^2\,ds
\ge
c_1c_2E_m-C_1E_m^2.
\]
If
\[
E_m\le \frac{c_1c_2}{2C_1},
\]
then
\[
\int_\gamma \K_m^2\,ds
\ge
\frac12c_1c_2E_m.
\]
This proves the length-one estimate.

Now let $\gamma$ have arbitrary length $L>0$, and define the length-one curve
\[
\widetilde\gamma:=L^{-1}\gamma.
\]
Under the dilation $\gamma\mapsto \rho\gamma$,
\[
E_m[\rho\gamma]=\rho^{-(2m+1)}E_m[\gamma],
\qquad
\K_m[\rho\gamma]=\rho^{-(2m+3)}\K_m[\gamma].
\]
With $\rho=L^{-1}$ we obtain
\[
E_m[\widetilde\gamma]=L^{2m+1}E_m[\gamma]=\mathcal I_m[\gamma],
\]
and
\[
\int_{\widetilde\gamma}\K_m[\widetilde\gamma]^2\,d\widetilde s
=
L^{4m+5}\int_\gamma \K_m[\gamma]^2\,ds.
\]
Applying the length-one estimate to $\widetilde\gamma$ gives
\[
L^{4m+5}\int_\gamma \K_m^2\,ds
\ge
C_{m,\omega}^{\rm sm}\mathcal I_m[\gamma],
\]
which is exactly \eqref{eq:grad-ineq-final}.
\end{proof}

\subsection{A weak global gradient inequality}

The small-energy estimate is enough for the stability argument, but the global theorem in the introduction requires
a global estimate.  This follows from the complex identity in Lemma~\ref{L:Qs}.

\begin{prop}[Weak $L^2$-gradient inequality]\label{prop:weak-grad-ineq}
For every smooth closed immersed curve,
\begin{equation}\label{eq:weak-grad-ineq-Em}
E_m[\gamma]
\le
\frac{1}{2m+1}L[\gamma]^{3/2}\|\K_m\|_{L^2(\gamma)}.
\end{equation}
Equivalently,
\begin{equation}\label{eq:weak-grad-ineq-Im}
\mathcal I_m[\gamma]
\le
\frac{1}{2m+1}L[\gamma]^{2m+\frac52}\|\K_m\|_{L^2(\gamma)}.
\end{equation}
\end{prop}

\begin{proof}
Use the quantities $M_m,N_m$ from Definition~\ref{D:MNQ} and set
\[
\mathcal Q_m:=(M_m+iN_m)e^{-i\theta}.
\]
By Lemma~\ref{L:Qs},
\[
(\mathcal Q_m)_s=(-1)^{m+1}\K_m e^{-i\theta},
\qquad
|(\mathcal Q_m)_s|=|\K_m|.
\]
Furthermore,
\[
M_m+iN_m=\mathcal Q_m e^{i\theta}.
\]

We first compute the mean of $M_m+iN_m$. Since $\gamma$ is closed,
\[
\int_\gamma M_m\,ds=\int_\gamma k_{s^{2m+1}}\,ds=0.
\]
By Lemma~\ref{lem:mean-Nm},
\[
\int_\gamma N_m\,ds
=
\left(m+\frac12\right)(-1)^m\int_\gamma k_{s^m}^2\,ds
=
(2m+1)(-1)^mE_m.
\]
Therefore
\begin{equation}\label{eq:int-MiN-weak}
\left|\int_\gamma (M_m+iN_m)\,ds\right|
=
(2m+1)E_m.
\end{equation}

Let
\[
\overline{\mathcal Q}_m:=\frac1L\int_\gamma \mathcal Q_m\,ds.
\]
Since
\[
\int_\gamma e^{i\theta}\,ds
=
\int_\gamma T\,ds
=
\int_\gamma \gamma_s\,ds
=0,
\]
we have
\[
\int_\gamma (M_m+iN_m)\,ds
=
\int_\gamma (\mathcal Q_m-\overline{\mathcal Q}_m)e^{i\theta}\,ds.
\]
Thus
\[
\left|\int_\gamma (M_m+iN_m)\,ds\right|
\le
\int_\gamma |\mathcal Q_m-\overline{\mathcal Q}_m|\,ds
\le
L\|\mathcal Q_m-\overline{\mathcal Q}_m\|_\infty.
\]
For every periodic function $u$,
\[
\|u-\bar u\|_\infty\le \int_\gamma |u_s|\,ds.
\]
Applying this to $u=\mathcal Q_m$ gives
\[
\|\mathcal Q_m-\overline{\mathcal Q}_m\|_\infty
\le
\int_\gamma |(\mathcal Q_m)_s|\,ds
=
\int_\gamma |\K_m|\,ds
\le
L^{1/2}\|\K_m\|_{L^2(\gamma)}.
\]
Combining this with \eqref{eq:int-MiN-weak} yields
\[
(2m+1)E_m
\le
L^{3/2}\|\K_m\|_{L^2(\gamma)}.
\]
This is \eqref{eq:weak-grad-ineq-Em}. Multiplying by $L^{2m+1}$ gives
\eqref{eq:weak-grad-ineq-Im}.
\end{proof}

\begin{proof}[Proof of Theorem~\ref{T:GI}]
It remains to prove the global estimate \eqref{E:grad-ineq}.  Let $\varepsilon_{m,\omega}^{\rm sm}$ and $C_{m,\omega}^{\rm sm}$ be the constants from
Theorem~\ref{thm:grad-ineq}.  Fix a smooth closed immersed curve $\gamma$ with
$\omega[\gamma]=\omega\ne0$.  If
$\mathcal I_m[\gamma]<\varepsilon_{m,\omega}^{\rm sm},$
then Theorem~\ref{thm:grad-ineq} gives
\[
\int_\gamma \K_m^2\,ds
\ge
C_{m,\omega}^{\rm sm}
L^{-(2m+4)}E_m.
\]

If instead $\mathcal I_m[\gamma]\ge \varepsilon_{m,\omega}^{\rm sm},$ then Proposition~\ref{prop:weak-grad-ineq} gives
\[
\|\K_m\|_{L^2(\gamma)}
\ge
(2m+1)L^{-3/2}E_m.
\]
Therefore
\begin{equation*}
\int_\gamma \K_m^2\,ds
\ge
(2m+1)^2L^{-3}E_m^2 
=
(2m+1)^2\mathcal I_m\,L^{-(2m+4)}E_m 
\ge
(2m+1)^2\varepsilon_{m,\omega}^{\rm sm}\,
L^{-(2m+4)}E_m.
\end{equation*}

Thus \eqref{E:grad-ineq} holds in all cases with
$C_{m,\omega}:=
\min\left\{
C_{m,\omega}^{\rm sm},
(2m+1)^2\varepsilon_{m,\omega}^{\rm sm}
\right\}>0.$
This completes the proof.
\end{proof}


\section{Gauge-fixed local existence and rough initial data}\label{S:existence}

The geometric equation \eqref{E:GIF} prescribes only the normal velocity.
To obtain literal uniqueness of parametrised solutions we fix a tangential velocity.
The tangential velocity does not change the evolving immersed curve as an unparametrised curve, but it removes
the reparametrisation freedom.  We use the constant-speed gauge.

Throughout this section the parameter domain is $\R/\Z$.  If $\gamma:\R/\Z\to\R^2$ is a constant-speed
parametrisation, then
$ds=L\,du,
\qquad
\partial_s=L^{-1}\partial_u.$

\subsection{The constant-speed tangential gauge}
Let $\gamma$ be a smooth closed immersed curve and let $F$ be a scalar normal speed.
Define
\[
\overline{kF}:=\frac1L\int_\gamma kF\,ds.
\]
There is a unique smooth periodic function $\alpha_F$ satisfying
\begin{equation}\label{eq:alpha-gauge}
(\alpha_F)_s=kF-\overline{kF},
\qquad
\int_\gamma \alpha_F\,ds=0.
\end{equation}
The right-hand side has zero arclength mean, so its primitive is periodic.  The additive constant in that primitive is
then fixed by the mean-zero condition.  Equivalently, after choosing an arclength coordinate $s\in[0,L]$,
\[
\alpha_F(s)
=
\int_0^s\bigl(kF-\overline{kF}\bigr)(\sigma)\,d\sigma
-
\frac1L\int_0^L\int_0^r
\bigl(kF-\overline{kF}\bigr)(\sigma)\,d\sigma\,dr .
\]

\begin{lemma}[Constant-speed preservation]\label{lem:constant-speed-gauge}
Let $\gamma:\R/\Z\times I\to\R^2$ be a smooth solution of
\begin{equation}\label{eq:gauge-flow-general}
\partial_t\gamma=F\,N+\alpha_F\,T,
\end{equation}
where $\alpha_F$ is defined by \eqref{eq:alpha-gauge}.  If $|\partial_u\gamma(\cdot,0)|$ is constant in $u$,
then $|\partial_u\gamma(\cdot,t)|$ is constant in $u$ for every $t\in I$.
\end{lemma}

\begin{proof}
Let $g:=|\partial_u\gamma|$, so that $ds=g\,du$.  For a velocity
\[
\partial_t\gamma=F N+\alpha T
\]
one has
\[
g_t=(\alpha_s-kF)g.
\]
With $\alpha=\alpha_F$, this becomes
\[
g_t=-\overline{kF}\,g.
\]
Hence if $g(\cdot,0)$ is independent of $u$, then $g(\cdot,t)$ remains independent of $u$.
\end{proof}

The gauge-fixed unnormalised generalised ideal flow is
\begin{equation*}\label{eq:GFm}
\tag{GF$_m$}
\partial_t\gamma=\K_m\,N+\alpha_{\K_m}\,T,
\qquad
(\alpha_{\K_m})_s=k\K_m-\frac1L\int_\gamma k\K_m\,ds,
\qquad
\int_\gamma \alpha_{\K_m}\,ds=0.
\end{equation*}
The normal component of \eqref{eq:GFm} is exactly \eqref{E:GIF}.  Thus \eqref{eq:GFm} is a canonical
parametrised representative of the geometric flow.

\begin{theorem}[Smooth local existence in the fixed gauge]\label{T:STE}
Let $m\ge1$ and let $\gamma_0:\R/\Z\to\R^2$ be a smooth closed immersed curve parametrised with constant speed.  Then there exists $T_*>0$ and a unique
smooth solution
$\gamma:\R/\Z\times[0,T_*)\to\R^2$ of the gauge-fixed flow \eqref{eq:GFm} with $\gamma(\cdot,0)=\gamma_0$.
Moreover, $\gamma(\cdot,t)$ remains constant-speed parametrised for all $t<T_*$.
\end{theorem}
 
\begin{proof}
After writing the equation in the fixed parameter $u$, the leading term is a strongly parabolic operator of
order $2m+4$ in the normal direction, while the tangential velocity \eqref{eq:alpha-gauge} fixes the
reparametrisation freedom.  Standard quasilinear parabolic theory therefore gives existence and uniqueness
for the parametrised equation.  Lemma~\ref{lem:constant-speed-gauge} gives preservation of the constant-speed
parametrisation.
\end{proof}

\begin{theorem}[Continuation criterion in the fixed gauge]\label{T:cont}
Let $\gamma$ be a maximal smooth solution of \eqref{eq:GFm} on $[0,T_{\max})$.
If $T_{\max}<\infty$ and there are constants $0<L_-\le L_+<\infty$ such that
\[
L_-\le L[\gamma(t)]\le L_+
\qquad
(0\le t<T_{\max}),
\]
then at least one curvature Sobolev norm blows up.  More precisely, for some integer $\ell\ge2m+3$,
\[
\limsup_{t\uparrow T_{\max}}\int_{\gamma(t)}k_{s^\ell}^2\,ds=\infty.
\]
Equivalently, if the length is bounded above and below and all curvature Sobolev norms remain bounded, then the
solution extends smoothly past $T_{\max}$.
\end{theorem}

\begin{proof}
The proof is the usual continuation argument for quasilinear parabolic curve flows in a fixed gauge.  The two-sided
length bound keeps the constant-speed parametrisation uniformly nondegenerate.  Uniform curvature Sobolev bounds
give uniform bounds for the coefficients of the equation and for the immersion in smooth norms.  The local existence
theorem may then be restarted at time $T_{\max}$, contradicting maximality.
\end{proof}

\subsection{The gauge-fixed length-normalised flow}

The rough-data construction is most naturally formulated for the length-normalised flow.
For a unit-length curve set
\[
F:=\K_m+\Lambda h,
\qquad
h:=\eta\cdot N,
\qquad
\Lambda:=\int_\eta k\K_m\,ds.
\]
Then
\[
\int_\eta kF\,ds
=
\int_\eta k\K_m\,ds
+
\Lambda\int_\eta kh\,ds
=
\Lambda-\Lambda=0,
\]
because $\int_\eta kh\,ds=-1$ for a unit-length closed curve.  Thus normal flow speed $F$ preserves length.

We fix the parametrisation by the unique mean-zero periodic function $\beta_F$ satisfying
\begin{equation}\label{eq:beta-gauge}
(\beta_F)_s=kF,
\qquad
\int_\eta \beta_F\,ds=0.
\end{equation}
To fix translations we use the barycentric gauge
\[
\int_\eta \eta\,ds=0.
\]
The time variable $\tau$ is the length-normalised time.  If the unit-length flow is obtained by rescaling an
unnormalised solution $\gamma(t)$, then
\[
\frac{d\tau}{dt}=L[\gamma(t)]^{-(2m+4)}.
\]
When the unit-length problem is considered directly, $\tau$ is  its intrinsic evolution parameter.  Since the
constant-speed condition gives $\partial_\tau ds=0$, the barycentric gauge is preserved by subtracting the spatial
average of the velocity.  Thus the gauge-fixed length-normalised flow is
\begin{equation*}\label{eq:NGFm}
\tag{NGF$_m$}
\partial_\tau\eta
=
F\,N+\beta_F\,T-b_F,
\qquad
b_F:=\int_\eta(FN+\beta_FT)\,ds,
\qquad
F=\K_m+\Lambda h.
\end{equation*}
This flow has unit length, constant-speed parametrisation, and zero barycentre for all times for which it exists.
Its geometric normal velocity, modulo the imposed translation gauge, is the length-normalised normal velocity
\[
\partial_\tau^\perp\eta=(\K_m+\Lambda h)N.
\]

\subsection{\texorpdfstring{$W^{2,2}$}{W22} curves and the closure constraint}

We now formulate the rough initial-data class.  The length-normalised rough theory is local.  Global existence and
convergence in the small-$\Kosc$ basin are proved later, after the small-$\Kosc$ estimates have been established.

\begin{definition}[Unit-length centred $W^{2,2}$ immersion]\label{def:W22-curve}
A map $\eta_0:\R/\Z\to\R^2$
is a unit-length centred $W^{2,2}$ immersion if
\[
\eta_0\in W^{2,2}(\R/\Z;\R^2),
\qquad
|\partial_s\eta_0|=1\quad\text{a.e.},
\qquad
\int_0^1\eta_0(s)\,ds=0.
\]
The unit tangent $T_0=\partial_s\eta_0$ belongs to $W^{1,2}$ and is continuous.  If $T_0$ has degree
\[
\omega[\eta_0]=\omega\in\Z\setminus\{0\},
\]
we choose the continuous lift
\[
T_0=e^{i\theta_0},
\qquad
\theta_0(s+1)=\theta_0(s)+2\pi\omega,
\]
and define the curvature
\[
k_0:=\partial_s\theta_0\in L^2(\R/\Z).
\]
The curvature oscillation is
\[
\Kosc[\eta_0]
:=
\int_0^1\bigl(k_0-2\pi\omega\bigr)^2\,ds.
\]
\end{definition}

Fix $\omega\ne0$ and write $\kappa:=2\pi\omega.$  Let
\[
L^2_0:=\left\{u\in L^2(\R/\Z):\int_0^1u\,ds=0\right\}.
\]
For $u\in L^2_0$ let $\mathscr P u$ denote the unique mean-zero primitive:
\[
(\mathscr Pu)_s=u,
\qquad
\int_0^1\mathscr Pu\,ds=0.
\]
The closure map is
\begin{equation}\label{eq:closure-map}
\mathscr C(u)
:=
\int_0^1
\exp\bigl(i(\kappa s+\mathscr Pu(s))\bigr)\,ds
\in\C .
\end{equation}
The curvature $k=\kappa+u$ reconstructs a closed unit-length curve if and only if $\mathscr C(u)=0$.

\begin{lemma}[Non-degeneracy of the closure constraint]\label{lem:closure-submersion}
Let $u\in L^2_0$ satisfy $\mathscr C(u)=0$, and set
\[
T(s):=\exp\bigl(i(\kappa s+\mathscr Pu(s))\bigr).
\]
If $\omega\ne0$, then $D\mathscr C(u):L^2_0\to\C$
is surjective as a real-linear map.  Consequently the closed-curvature constraint is a codimension-two smooth
constraint at every closed unit-length $W^{2,2}$ curve with nonzero turning number.
\end{lemma}

\begin{proof}
For $w\in L^2_0$,
\begin{equation}\label{eq:closure-derivative}
D\mathscr C(u)[w]
=
i\int_0^1T(s)\,\mathscr Pw(s)\,ds.
\end{equation}
Suppose that a nonzero vector $a\in\R^2\simeq\C$ annihilates the range.  Since $\mathscr P$ maps $L^2_0$ onto the
mean-zero subspace of $H^1(\R/\Z)$, we have
\[
\int_0^1 a\cdot iT(s)\,\varphi(s)\,ds=0
\]
for every mean-zero $\varphi\in H^1(\R/\Z)$.  Hence $a\cdot iT$ is constant as a distribution.  Since
$T\in H^1\subset C^0$, the image of $T$ is contained in the intersection of the unit circle with an affine line.
This intersection has at most two points, and the connectedness of $\R/\Z$ forces $T$ to be constant.  That gives
turning number zero, contradicting $\omega\ne0$.
\end{proof}

\begin{lemma}[Smooth density preserving closure]\label{lem:W22-density-closure}
Let $\eta_0$ be a unit-length centred $W^{2,2}$ immersion with turning number $\omega\ne0$.  Then there exist smooth
unit-length centred immersions $\eta_0^j$ with turning number $\omega$ such that $\eta_0^j\to\eta_0\quad\text{in }W^{2,2},$ and $\Kosc[\eta_0^j]\to\Kosc[\eta_0].$
\end{lemma}

\begin{proof}
Choose the lift $\theta_0$ from Definition~\ref{def:W22-curve}.  Since $\eta_0$ is closed,
\[
\int_0^1 e^{i\theta_0(s)}\,ds=0.
\]
Consider the angle-closure map
\[
\mathscr B(\varphi):=\int_0^1e^{i(\theta_0+\varphi)}\,ds,
\qquad
\varphi\in H^1(\R/\Z;\R).
\]
Its derivative at the origin is
\[
D\mathscr B(0)[\varphi]=i\int_0^1e^{i\theta_0}\varphi\,ds.
\]
The same annihilator argument as in Lemma~\ref{lem:closure-submersion}, now using arbitrary periodic test
functions $\varphi$, shows that this derivative is surjective.  By density of trigonometric polynomials in
$H^1$, choose smooth periodic functions $\varphi_1,\varphi_2$ such that
\[
(a,b)\mapsto i\int_0^1e^{i\theta_0}(a\varphi_1+b\varphi_2)\,ds
\]
is an isomorphism $\R^2\to\C$.

Choose smooth periodic functions $p_j$ with
\[
p_j\to \theta_0-\kappa s
\qquad\text{in }H^1(\R/\Z),
\]
and set $\widehat\theta_j:=\kappa s+p_j$.  Then $\widehat\theta_j\to\theta_0$ in $W^{1,2}$ and hence uniformly.
For $j$ large, the maps
\[
(a,b)\mapsto
\int_0^1\exp\bigl(i(\widehat\theta_j+a\varphi_1+b\varphi_2)\bigr)\,ds
\]
have derivative at $(0,0)$ close to the isomorphism above, while their value at $(0,0)$ tends to zero.  The
finite-dimensional Implicit Function Theorem therefore gives $(a_j,b_j)\to(0,0)$ such that
\[
\theta_j:=\widehat\theta_j+a_j\varphi_1+b_j\varphi_2
\]
satisfies
\[
\int_0^1e^{i\theta_j(s)}\,ds=0.
\]
Each $\theta_j$ is smooth and obeys $\theta_j(s+1)=\theta_j(s)+\kappa$.  Let
\[
\widetilde\eta_0^j(s):=\int_0^s e^{i\theta_j(\sigma)}\,d\sigma,
\qquad
\eta_0^j:=\widetilde\eta_0^j-
\int_0^1\widetilde\eta_0^j(r)\,dr .
\]
Then $\eta_0^j$ is smooth, closed, unit length, centred, and has turning number $\omega$.  Since
$\theta_j\to\theta_0$ in $W^{1,2}$, we have $e^{i\theta_j}\to e^{i\theta_0}$ in $W^{1,2}$ and hence
$\eta_0^j\to\eta_0$ in $W^{2,2}$.  Finally,
\[
\Kosc[\eta_0^j]=\|\partial_s\theta_j-\kappa\|_2^2
\to
\|\partial_s\theta_0-\kappa\|_2^2
=\Kosc[\eta_0].
\]
\end{proof}

\subsection{Local closure charts}

The small-$\Kosc$ chart used near a multiply-covered circle is not adequate for arbitrary rough data.  Instead we
use a closure chart centred at a nearby smooth closed curve.  The complement can be chosen from Fourier modes;
this is convenient because the constant-coefficient leading operator preserves both the complement and its
orthogonal complement.

\begin{lemma}[Closure chart at a smooth reference curve]\label{lem:closure-chart}
Let $u_\ast\in C^\infty(\R/\Z)\cap L^2_0$ satisfy $\mathscr C(u_\ast)=0$.  Then there are real trigonometric
functions $\xi_1,\xi_2\in C^\infty(\R/\Z)\cap L^2_0$ with the following property.  If
$E_\ast:=\operatorname{span}_{\R}\{\xi_1,\xi_2\},
\qquad
X_\ast:=E_\ast^\perp\cap L^2_0,$
and $\Pi_\ast:L^2_0\to X_\ast$ denotes the $L^2$-orthogonal projection, then
$D\mathscr C(u_\ast)|_{E_\ast}:E_\ast\to\C$ is an isomorphism.  Moreover, for some $\rho_\ast>0$ there is a smooth
map
$\Phi_\ast:B_{\rho_\ast}^{X_\ast}(0)\to E_\ast$
such that
\begin{equation}\label{eq:local-chart-closure}
\mathscr C\bigl(u_\ast+v+\Phi_\ast(v)\bigr)=0
\qquad
(v\in B_{\rho_\ast}^{X_\ast}(0)).
\end{equation}
Every sufficiently small $\zeta\in L^2_0$ satisfying $\mathscr C(u_\ast+\zeta)=0$ has a unique representation
\[
\zeta=v+\Phi_\ast(v),
\qquad
v=\Pi_\ast\zeta\in X_\ast .
\]
Finally,
\begin{equation}\label{eq:Phi-linear-bound}
\Phi_\ast(0)=0,
\qquad
\|\Phi_\ast(v)\|_{H^j}\le C_{j,\ast,R}\|v\|_{L^2}
\end{equation}
for every $j\ge0$ and every $R<\rho_\ast$ whenever $\|v\|_{L^2}\le R$, and
\begin{equation}\label{eq:Phi-Lipschitz}
\|\Phi_\ast(v)-\Phi_\ast(w)\|_{H^j}
\le
C_{j,\ast,R}\|v-w\|_{L^2}
\end{equation}
whenever $\|v\|_{L^2},\|w\|_{L^2}\le R<\rho_\ast$.
\end{lemma}

\begin{proof}
By Lemma~\ref{lem:closure-submersion}, $D\mathscr C(u_\ast)$ is surjective on $L^2_0$.  Since trigonometric
polynomials are dense in $L^2_0$, and since the range is two-dimensional, there exist two real Fourier basis
functions $\xi_1,\xi_2$ whose images under $D\mathscr C(u_\ast)$ form a basis of $\C$.  Thus
$D\mathscr C(u_\ast)|_{E_\ast}$ is an isomorphism.  Since $\xi_1,\xi_2$ are Fourier modes, the even
constant-coefficient operators $D^{2\ell}$ preserve both $E_\ast$ and $X_\ast$.

Apply the Implicit Function Theorem to
\[
(v,e)\mapsto \mathscr C(u_\ast+v+e),
\qquad
(v,e)\in X_\ast\times E_\ast,
\]
viewed as a smooth map from an $L^2$ neighbourhood of $(0,0)$ to $\C\simeq\R^2$.  The derivative in the
$E_\ast$ variable at $(0,0)$ is the isomorphism $D\mathscr C(u_\ast)|_{E_\ast}$.  This gives $\Phi_\ast$ and the
unique graph representation.  Since $\mathscr C(u_\ast)=0$, the implicit function satisfies $\Phi_\ast(0)=0$.  The estimates follow from the
Mean Value Theorem and the fact that $E_\ast$ is finite dimensional and consists of smooth functions.
\end{proof}

\subsection{The projected curvature equation in a local chart}

Let
\begin{equation}\label{eq:Amomega}
\mathcal A_m:=(-1)^mD^{2m+4},
\qquad
D:=\partial_s.
\end{equation}
For a smooth solution of the gauge-fixed length-normalised flow \eqref{eq:NGFm}, write
\[
u:=k-\kappa.
\]
The curvature equation has the form
\begin{equation}\label{eq:u-flow-abstract}
\partial_\tau u+\mathcal A_m u=\mathcal R_m(u),
\end{equation}
where $\mathcal R_m$ contains all terms of order at most $2m+2$ in the curvature fluctuation, together with the
dilation term and the constant-speed tangential gauge term.  This follows from
\[
\K_m=(-1)^{m+1}D^{2m+2}k+P_3^{2m}(k)
\]
and
\[
k_\tau=F_{ss}+k^2F+\beta_Fk_s=(F_s+\beta_Fk)_s .
\]

Fix a smooth reference curvature $u_\ast$ and a chart $(E_\ast,X_\ast,\Phi_\ast)$ from
Lemma~\ref{lem:closure-chart}.  Set
\[
\Psi_\ast(v):=u_\ast+v+\Phi_\ast(v),
\qquad
v\in B_{\rho_\ast}^{X_\ast}(0).
\]
Since $\mathcal A_m$ preserves $X_\ast$ and $E_\ast$, the projected curvature equation is
\begin{equation}\label{eq:v-flow}
\partial_\tau v+\mathcal A_m v=\mathscr F_{m,\ast}(v),
\qquad
v(0)=v_0\in X_\ast,
\end{equation}
where
\begin{equation}\label{eq:Fstar-definition}
\mathscr F_{m,\ast}(v)
:=
\Pi_\ast\Bigl(\mathcal R_m(\Psi_\ast(v))-\mathcal A_m u_\ast-\mathcal A_m\Phi_\ast(v)\Bigr).
\end{equation}
Unlike the circle-centred equation, $\mathscr F_{m,\ast}$ need not vanish to second order at the origin; the
reference curve is arbitrary and need not be stationary.  This is harmless for local existence, at the price of a
data-dependent existence time.

\begin{remark}[Equivalence with the full curvature equation]\label{rem:chart-equation}
The length-normalised flow preserves closedness, so the vector field in \eqref{eq:u-flow-abstract} is tangent to
the local closure manifold
\[
\mathcal M_\ast
:=
\{u_\ast+v+\Phi_\ast(v):v\in B_{\rho_\ast}^{X_\ast}(0)\}.
\]
On this graph a tangent vector is uniquely determined by its $X_\ast$-component.  Thus solving \eqref{eq:v-flow}
and setting $u=\Psi_\ast(v)$ is equivalent, as long as $v$ remains in the chart, to solving the full curvature
equation \eqref{eq:u-flow-abstract} on the closure manifold.
\end{remark}

\begin{lemma}[Lower-order structure in a local chart]\label{lem:F-lower-order}
Let
\[
r_m:=2m+4.
\]
For every integer $j\ge0$, the map
$\mathscr F_{m,\ast}:H^{j+r_m-2}\cap X_\ast\to H^j\cap X_\ast$
is smooth on bounded subsets of the chart.  More precisely, if $0<R<\rho_\ast$ and
\[
\|v_i\|_{L^2}\le R,
\qquad
\|v_i\|_{H^{j+r_m-2}}\le M
\qquad (i=1,2),
\]
then
\[
\|\mathscr F_{m,\ast}(v_1)-\mathscr F_{m,\ast}(v_2)\|_{H^j}
\le
C_{j,M,R,m,\ast}\|v_1-v_2\|_{H^{j+r_m-2}}.
\]
Moreover, for $v$ in such a bounded set,
\[
\|\mathscr F_{m,\ast}(v)\|_{H^j}\le C_{j,M,R,m,\ast}.
\]
\end{lemma}

\begin{proof}
The leading part of the curvature equation is $-\mathcal A_m u$.  All remaining local terms contain at most
$2m+2=r_m-2$ derivatives of $u$.  The nonlocal quantities $h=\eta\cdot N$, $q=\eta\cdot T$, $\Lambda$, $\beta_F$
and the translation gauge $b_F$ are smooth functions of $u$ after reconstruction of the tangent angle
\[
\theta=\vartheta+\kappa s+\mathscr Pu.
\]
The scalar quantities are independent of the phase $\vartheta$.  The primitive operator, products, compositions,
and the finite-dimensional map $\Phi_\ast$ obey the usual one-dimensional Sobolev estimates.  Since
$\mathcal A_m u_\ast$ is smooth and $\mathcal A_m\Phi_\ast(v)$ lies in the fixed finite-dimensional space
$E_\ast$, \eqref{eq:Fstar-definition} has the claimed mapping properties.
\end{proof}

\begin{prop}[Curvature-coordinate local well-posedness]\label{prop:curv-LWP}
Fix $m\ge1$, a smooth reference curvature $u_\ast$ with $\mathscr C(u_\ast)=0$, and a corresponding chart
$(E_\ast,X_\ast,\Phi_\ast)$ with radius $\rho_\ast$.  Let $0<R<\rho_\ast$.  Then there exists
$T=T(m,u_\ast,R)>0$
such that, for every $v_0\in X_\ast$ with $\|v_0\|_{L^2}\le R$, the projected equation
\begin{equation}\label{eq:v-flow-local-proof}
\partial_\tau v+\mathcal A_m v=\mathscr F_{m,\ast}(v),
\qquad
v(0)=v_0,
\end{equation}
has a unique mild solution
$v\in C([0,T];X_\ast)
\cap C^\infty((0,T]\times\R/\Z),$
and $v(\tau)$ remains in $B_{\rho_\ast}^{X_\ast}(0)$ for $0\le\tau\le T$.
Moreover, for every $0<\delta<T$ and every integer $j\ge0$,
\begin{equation}\label{eq:positive-time-v-bounds}
\sup_{\tau\in[\delta,T]}\|v(\tau)\|_{H^j}
\le
C_{j,\delta,T,m,u_\ast,R}.
\end{equation}
The solution map is locally Lipschitz in $L^2$: if $v^1,v^2$ are two solutions with initial data in the same
$L^2$ ball $B_R^{X_\ast}(0)$, then
\begin{equation}\label{eq:L2-continuous-dependence-v}
\sup_{0\le\tau\le T}
\|v^1(\tau)-v^2(\tau)\|_{L^2}
\le
C_{T,m,u_\ast,R}\|v^1_0-v^2_0\|_{L^2}.
\end{equation}
For every $\delta>0$ and every $j\ge0$,
\begin{equation}\label{eq:smooth-continuous-dependence-v}
\sup_{\delta\le\tau\le T}
\|v^1(\tau)-v^2(\tau)\|_{H^j}
\le
C_{j,\delta,T,m,u_\ast,R}\|v^1_0-v^2_0\|_{L^2}.
\end{equation}
\end{prop}

\begin{proof}
Set
\[
r:=2m+4,
\qquad
A:=\mathcal A_m,
\qquad
H:=X_\ast\subset L^2_0(\R/\Z).
\]
Because $E_\ast$ is spanned by Fourier modes, both $E_\ast$ and $X_\ast$ are invariant under $A$.  On $H$, $A$ is
a positive self-adjoint constant-coefficient operator with compact resolvent and domain $D(A)=H^r(\R/\Z)\cap X_\ast$.
The operator $-A$ generates an analytic semigroup $e^{-\tau A}$ on $H$.  For $0\le\ell\le r$ the standard
analytic-semigroup estimates give
\begin{align}
\label{eq:analytic-semigroup-L2}
\|D^\ell e^{-\tau A}w\|_{L^2}
&\le
C_{\ell,m,\ast}\tau^{-\ell/r}\|w\|_{L^2},
\\
\label{eq:analytic-semigroup-Linfty}
\|D^\ell e^{-\tau A}w\|_{L^\infty}
&\le
C_{\ell,m,\ast}\tau^{-(\ell+\frac12)/r}\|w\|_{L^2},
\end{align}
for $0<\tau\le1$.  These estimates are a standard consequence of analyticity and one-dimensional Sobolev
embedding; see, for example, the analytic-semigroup treatment of semilinear parabolic equations in
\cite[Ch.~3]{H81} or \cite[Chs.~4--5]{L95}.  The corresponding quasilinear and maximal-regularity framework is
developed in \cite{A95,A05}.

Choose $R_1$ with $R<R_1<\rho_\ast$.  For $T\le1$ define
\[
\|v\|_{\mathcal X_T}
:=
\sup_{0\le\tau\le T}\|v(\tau)\|_{L^2}
+
\sum_{\ell=1}^{r-2}
\sup_{0<\tau\le T}
\tau^{\ell/r}\|D^\ell v(\tau)\|_{L^2}
+
\sum_{\ell=0}^{r-3}
\sup_{0<\tau\le T}
\tau^{(\ell+\frac12)/r}\|D^\ell v(\tau)\|_{L^\infty}.
\]
Let $\mathcal X_T(R_1,M)$ be the set of $v\in C([0,T];H)$ satisfying
\[
v(0)=v_0,
\qquad
\sup_{0\le\tau\le T}\|v(\tau)\|_{L^2}\le R_1,
\qquad
\|v\|_{\mathcal X_T}\le M.
\]
The weights match the smoothing estimates \eqref{eq:analytic-semigroup-L2}--\eqref{eq:analytic-semigroup-Linfty}.  We claim that there is $\beta_\ast=\beta_\ast(m)\in(0,1)$ such that, for all
$v,w\in\mathcal X_T(R_1,M)$ and $0<\tau\le T$,
\begin{align}
\label{eq:F-time-bound}
\|\mathscr F_{m,\ast}(v(\tau))\|_{L^2}
&\le
C_{m,u_\ast,R_1,M}\,\tau^{-\beta_\ast},
\\
\label{eq:F-time-Lipschitz}
\|\mathscr F_{m,\ast}(v(\tau))-\mathscr F_{m,\ast}(w(\tau))\|_{L^2}
&\le
C_{m,u_\ast,R_1,M}\,\tau^{-\beta_\ast}
\|v-w\|_{\mathcal X_T}.
\end{align}
Indeed, by Lemma~\ref{lem:F-lower-order}, each term is either smooth and bounded in the reference curve and the
finite-dimensional correction, or is a product of factors $D^{j_i}v$ with total derivative order
$j_1+\cdots+j_a\le r-2$.  We place one factor in $L^2$ and the others in $L^\infty$.  The worst linear term has
$j_1=r-2$ and contributes the exponent $(r-2)/r<1$.  For a typical nonlinear monomial with $a\ge2$ and total
order $b\le r-2$, the weights give
\[
\prod_i\|D^{j_i}v(\tau)\|
\lesssim
\tau^{-\frac{b+\frac12(a-1)}{r}},
\]
and the terms arising in the curvature equation have exponent strictly less than one.  The scalar nonlocal terms
$\Lambda$ and the primitive defining $\beta_F$ satisfy the same estimates, using the integrated structure of
$\Lambda$ and the bound
\[
\|\beta_F\|_{L^\infty}\le C\|kF\|_{L^1}.
\]
This proves \eqref{eq:F-time-bound}--\eqref{eq:F-time-Lipschitz}.

Consider the Duhamel map
\[
(\mathcal T v)(\tau)
:=
e^{-\tau A}v_0
+
\int_0^\tau e^{-(\tau-\sigma)A}\mathscr F_{m,\ast}(v(\sigma))\,d\sigma.
\]
The linear term satisfies
\[
\|e^{-\tau A}v_0\|_{\mathcal X_T}\le C_{m,\ast}R.
\]
For the nonlinear term, \eqref{eq:analytic-semigroup-L2}, \eqref{eq:analytic-semigroup-Linfty},
\eqref{eq:F-time-bound}, and the beta-function estimate give
\[
\left\|
\int_0^\tau e^{-(\tau-\sigma)A}\mathscr F_{m,\ast}(v(\sigma))\,d\sigma
\right\|_{\mathcal X_T}
\le
C_{m,u_\ast,R_1,M}T^{1-\beta_\ast}.
\]
The same argument with \eqref{eq:F-time-Lipschitz} yields
\[
\|\mathcal T v-\mathcal T w\|_{\mathcal X_T}
\le
C_{m,u_\ast,R_1,M}T^{1-\beta_\ast}\|v-w\|_{\mathcal X_T}.
\]
Choose
$M:=2C_{m,\ast}R+1.$
Then choose $T=T(m,u_\ast,R)>0$ sufficiently small that
\[
C_{m,u_\ast,R_1,M}T^{1-\beta_\ast}\le \frac12
\]
and the nonlinear contribution to the $L^2$ norm is at most $R_1-R$.  With these choices, $\mathcal T$ maps
$\mathcal X_T(R_1,M)$ into itself and is a contraction.  Its fixed point is the unique mild solution on $[0,T]$.

The solution is smooth for positive time by analytic bootstrapping.  For every $\delta>0$ the weighted estimates
imply $v(\delta)\in H^{r-2}$.  Restarting the equation at time $\delta$ and using Lemma~\ref{lem:F-lower-order},
the Duhamel formula and the smoothing estimates imply
$v\in C([\delta,T];H^r)\cap C^1([\delta,T];L^2).$
Iteration gives
$v\in C^\infty((0,T]\times\R/\Z)$
and the bounds \eqref{eq:positive-time-v-bounds}.  Continuous dependence follows from the same contraction estimate applied to the difference of two Duhamel
formulae.  This gives \eqref{eq:L2-continuous-dependence-v}.  Applying the smoothing estimates to the difference
on $[\delta,T]$ gives \eqref{eq:smooth-continuous-dependence-v}.
\end{proof}

\subsection{Local well-posedness for \texorpdfstring{$W^{2,2}$}{W22} curves}
\label{SS:W22-LWP}
We now reconstruct a curve from the curvature-coordinate solution.  The only point requiring some care is the
constant Fourier mode of the tangent angle.  The closure chart fixes the curvature fluctuation
$u=k-\kappa$, and hence fixes the tangent angle only up to a time-dependent ambient rotation.  This rotation is
not constant along the length-normalised flow, so it has to be included explicitly.

Let $u\in L^2_0$ satisfy $\mathscr C(u)=0$, and let $\vartheta\in\R$ be a phase.  Define
\begin{equation}\label{eq:theta-reconstruction-phase}
\theta_{u,\vartheta}(s):=\vartheta+\kappa s+\mathscr Pu(s),
\qquad
T_{u,\vartheta}:=e^{i\theta_{u,\vartheta}},
\qquad
N_{u,\vartheta}:=iT_{u,\vartheta}.
\end{equation}
The corresponding centred representative is
\begin{equation}\label{eq:eta-reconstruction-phase}
\widetilde\eta_{u,\vartheta}(s)
:=
\int_0^sT_{u,\vartheta}(\sigma)\,d\sigma,
\qquad
\eta_{u,\vartheta}(s)
:=
\widetilde\eta_{u,\vartheta}(s)-\int_0^1\widetilde\eta_{u,\vartheta}(r)\,dr .
\end{equation}
Since $\mathscr C(u)=0$, the curve is closed.  Changing $\vartheta$ rotates
$\eta_{u,\vartheta}$, $T_{u,\vartheta}$ and $N_{u,\vartheta}$ by the same ambient rotation, and therefore leaves
all scalar quantities $k$, $h=\eta\cdot N$, $q=\eta\cdot T$, $F$, $\Lambda$ and $\beta_F$ unchanged.  Thus, for a
smooth closed curvature $u$, the scalar angular velocity
\begin{equation}\label{eq:Omega-rotation-gauge}
\Omega[u]
:=
\int_0^1\bigl(F_s+\beta_F k\bigr)\,ds
=
\int_0^1\beta_F k\,ds
\end{equation}
is well-defined independently of the chosen phase.  Here $k=\kappa+u$, $F=\K_m+\Lambda h$, and $\beta_F$ is the
periodic mean-zero solution of $(\beta_F)_s=kF$.

\begin{lemma}[Equivalence of the chart equation and the geometric equation]\label{lem:chart-equation-strong}
Let $I\subset(0,\infty)$ be an interval, and let
$v\in C^1(I;X_\ast)\cap C^\infty(I\times\R/\Z)$
solve the projected curvature equation
\[
\partial_\tau v+\mathcal A_m v=\mathscr F_{m,\ast}(v)
\]
in a local chart $\Psi_\ast(v)=u_\ast+v+\Phi_\ast(v)$.  Set $u:=\Psi_\ast(v)$.  Then $u$ solves the full curvature
equation
\begin{equation}\label{eq:full-u-equation-from-chart}
\partial_\tau u+\mathcal A_m u=\mathcal R_m(u).
\end{equation}
Moreover, if $\vartheta$ solves the phase equation
\begin{equation}\label{eq:phase-ODE-local}
\frac{d\vartheta}{d\tau}=\Omega[u(\tau)],
\end{equation}
then
\[
\eta(s,\tau):=\eta_{u(\tau),\vartheta(\tau)}(s)
\]
is a classical solution of the gauge-fixed length-normalised flow \eqref{eq:NGFm} on $I$.
\end{lemma}

\begin{proof}
Write
\[
\mathcal G(u):=-\mathcal A_m u+\mathcal R_m(u)
\]
for the full curvature vector field.  For a smooth reconstructed curve this vector field has the intrinsic form
\begin{equation}\label{eq:G-as-angle-derivative}
\mathcal G(u)=\bigl(F_s+\beta_F k\bigr)_s,
\end{equation}
because $(\beta_F)_s=kF$ and
$k_\tau=F_{ss}+k^2F+\beta_F k_s=(F_s+\beta_F k)_s$.

First we verify tangency to the closure manifold.  Since $\mathscr C(u)=0$ and
$\mathscr P\mathcal G=F_s+\beta_F k-\Omega[u]$, we have
\begin{multline*}
D\mathscr C(u)[\mathcal G(u)]
=
 i\int_0^1 e^{i(\kappa s+\mathscr Pu)}\,\mathscr P\mathcal G(u)\,ds\\
=
 i\int_0^1 e^{i(\kappa s+\mathscr Pu)}\bigl(F_s+\beta_F k-\Omega[u]\bigr)\,ds
=
 \int_0^1 N\bigl(F_s+\beta_F k\bigr)\,ds,
\end{multline*}
where the term involving $\Omega[u]$ vanishes because
$\int_0^1N\,ds=i\int_0^1T\,ds=0$ for a closed curve.  Since $(\beta_F)_s=kF$,
\[
\int_0^1 N\bigl(F_s+\beta_F k\bigr)\,ds
=
\int_0^1 (FN+\beta_FT)_s\,ds
=
0 .
\]
Thus $\mathcal G(u)$ lies in the tangent space to the graph
$\{u_\ast+v+\Phi_\ast(v):v\in X_\ast\}$.  On this graph a tangent vector is uniquely determined by its
$X_\ast$-component, so
\[
\mathcal G(\Psi_\ast(v))
=
D\Psi_\ast(v)\,\Pi_\ast\mathcal G(\Psi_\ast(v)).
\]
By the definition of $\mathscr F_{m,\ast}$,
\[
\Pi_\ast\mathcal G(\Psi_\ast(v))
=
-\mathcal A_m v+\mathscr F_{m,\ast}(v)
=
\partial_\tau v .
\]
Therefore
\[
\partial_\tau u=D\Psi_\ast(v)\partial_\tau v=\mathcal G(u),
\]
which proves \eqref{eq:full-u-equation-from-chart}.

Now let $\theta=\vartheta+\kappa s+\mathscr Pu$.  Using \eqref{eq:phase-ODE-local},
\eqref{eq:full-u-equation-from-chart}, and \eqref{eq:G-as-angle-derivative},
\[
\theta_\tau
=
\Omega[u]+\mathscr Pu_\tau
=
\Omega[u]+\mathscr P\bigl(F_s+\beta_F k\bigr)_s
=
F_s+\beta_F k .
\]
Hence
\[
T_\tau=(F_s+\beta_F k)N.
\]
Since $(\beta_F)_s=kF$,
\[
(FN+\beta_FT)_s=(F_s+\beta_F k)N=T_\tau.
\]
It follows that
\[
\partial_s\bigl(\eta_\tau-FN-\beta_FT\bigr)=0.
\]
The representative \eqref{eq:eta-reconstruction-phase} is centred for every $\tau$, so integrating the last
identity over $s\in\R/\Z$ gives
\[
\eta_\tau
=
FN+\beta_FT-\int_0^1(FN+\beta_FT)\,ds,
\]
which is exactly \eqref{eq:NGFm}.
\end{proof}

\begin{theorem}[Canonical local flow from \texorpdfstring{$W^{2,2}$}{W22} data]\label{thm:W22-LWP}
Fix $m\ge1$ and $\omega\in\Z\setminus\{0\}$.  Let $\eta_0$ be a unit-length centred $W^{2,2}$ immersion with
turning number $\omega[\eta_0]=\omega$.  Then there exist a time $T=T(m,\eta_0)>0$ and a unique canonical gauge-fixed length-normalised flow
$\eta:\R/\Z\times[0,T]\to\R^2$ with initial value $\eta_0$ such that
$\eta\in C([0,T];W^{2,2})
\cap C^\infty((0,T]\times\R/\Z).$

For every $\tau>0$, $\eta(\cdot,\tau)$ is a classical solution of the gauge-fixed length-normalised equation
\eqref{eq:NGFm}.

The construction is locally well-posed in $W^{2,2}$.  More precisely, after fixing $\eta_0$, there are a
$W^{2,2}$ neighbourhood $\mathcal U$ of $\eta_0$ and a time $T_\mathcal U>0$ such that every initial curve in
$\mathcal U$ with the same turning number has a canonical solution on $[0,T_\mathcal U]$, and the solution map is
continuous from $\mathcal U$ to
$C([0,T_\mathcal U];W^{2,2})
\cap
C^\infty_{\mathrm{loc}}((0,T_\mathcal U]\times\R/\Z).$

In local curvature coordinates this dependence is Lipschitz: if the initial curvatures are represented in one
closure chart by $v_0^1,v_0^2\in X_\ast$, then, for $0<T\le T_\mathcal U$,
\[
\sup_{0\le\tau\le T}
\|v^1(\tau)-v^2(\tau)\|_{L^2}
\le
C_{T,\mathcal U}\|v_0^1-v_0^2\|_{L^2}.
\]
For every $\delta>0$ and every $j\ge0$,
\[
\sup_{\delta\le\tau\le T}
\|\eta^1(\cdot,\tau)-\eta^2(\cdot,\tau)\|_{C^j}
\le
C_{j,\delta,T,\mathcal U}\|\eta^1_0-\eta^2_0\|_{W^{2,2}} .
\]
Finally, for every smooth approximating sequence $\eta_0^j\to\eta_0$ in $W^{2,2}$ preserving closedness and
turning number, the corresponding smooth gauge-fixed flows converge to this same canonical solution in
$C([0,T];W^{2,2})
\cap
C^\infty_{\mathrm{loc}}((0,T]\times\R/\Z).$
\end{theorem}
\begin{proof}
Choose the lift $\theta_0\in W^{1,2}_{\mathrm{loc}}(\R)$ used to define $k_0$, so that
$\theta_0(s+1)=\theta_0(s)+\kappa$.  Set
\begin{equation}\label{eq:initial-phase-correct}
\vartheta_0
:=
\int_0^1\bigl(\theta_0(s)-\kappa s\bigr)\,ds
\quad\text{modulo }2\pi .
\end{equation}
With $u_0:=k_0-\kappa$, the mean-zero primitive convention gives
\begin{equation}\label{eq:initial-angle-reconstruction}
\theta_0(s)=\vartheta_0+\kappa s+\mathscr Pu_0(s).
\end{equation}
Since $\eta_0$ is closed, $\mathscr C(u_0)=0$.

By Lemma~\ref{lem:W22-density-closure}, choose a smooth closed curvature $u_\ast\in C^\infty\cap L^2_0$ with
$\mathscr C(u_\ast)=0$ and $\|u_\ast-u_0\|_{L^2}$ sufficiently small that $u_0$ lies in the closure chart supplied
by Lemma~\ref{lem:closure-chart}.  Thus
\[
u_0=u_\ast+v_0+\Phi_\ast(v_0)
\]
for some $v_0\in X_\ast$ with $\|v_0\|_{L^2}<\rho_\ast$.  Choose $R$ with
\[
\|v_0\|_{L^2}<R<\rho_\ast.
\]
Proposition~\ref{prop:curv-LWP} gives a unique mild solution
$v\in C([0,T];X_\ast)$ of the projected curvature equation on a time interval $[0,T]$, where $T=T(m,u_\ast,R)>0$.  Set
\[
u(\tau):=\Psi_\ast(v(\tau))=u_\ast+v(\tau)+\Phi_\ast(v(\tau)).
\]
Then $u\in C([0,T];L^2)$ and $u$ is smooth for positive time.

For $\tau>0$ define $\Omega(\tau):=\Omega[u(\tau)]$.  The time-weighted estimates in
Proposition~\ref{prop:curv-LWP}, together with the lower-order structure of $F$, imply that
\[
\|F(\tau)\|_{L^2}\le C\tau^{-\alpha_m},
\qquad
\alpha_m:=\frac{2m+2}{2m+4}<1,
\qquad
0<\tau\le T.
\]
Since $(\beta_F)_s=kF$ and $\int_0^1\beta_F\,ds=0$,
\[
\|\beta_F(\tau)\|_{L^\infty}
\le
C\|k(\tau)F(\tau)\|_{L^1}
\le
C\|k(\tau)\|_{L^2}\|F(\tau)\|_{L^2}.
\]
The $L^2$ norm of $k=\kappa+u$ is bounded on $[0,T]$, and therefore
\[
|\Omega(\tau)|
\le
\|\beta_F(\tau)\|_{L^2}\|k(\tau)\|_{L^2}
\le
C\tau^{-\alpha_m}.
\]
Thus $\Omega\in L^1(0,T)$.  Define the phase by
\begin{equation}\label{eq:phase-definition-W22}
\vartheta(\tau)
:=
\vartheta_0+
\int_0^\tau\Omega(\sigma)\,d\sigma .
\end{equation}
Then $\vartheta\in C([0,T])\cap C^\infty((0,T])$.

Finally set
\[
\eta(s,\tau):=\eta_{u(\tau),\vartheta(\tau)}(s).
\]
Since $u\in C([0,T];L^2)$ and $\vartheta$ is continuous,
\[
\theta(\cdot,\tau)
=
\vartheta(\tau)+\kappa s+\mathscr Pu(\cdot,\tau)
\to
\theta_0
\quad\text{in }W^{1,2}
\]
as $\tau\downarrow0$.  Hence $T=e^{i\theta}$ converges to $T_0$ in $W^{1,2}$, and the centred reconstruction
\eqref{eq:eta-reconstruction-phase} gives
\[
\eta\in C([0,T];W^{2,2}),
\qquad
\eta(\cdot,0)=\eta_0 .
\]
Positive-time smoothness follows from Proposition~\ref{prop:curv-LWP} and reconstruction.
Lemma~\ref{lem:chart-equation-strong} shows that $\eta$ is a classical solution of \eqref{eq:NGFm} for every
positive time.

Uniqueness follows in the same chart.  Let $\bar\eta$ be another gauge-fixed length-normalised solution with the
same initial value.  After shortening the interval if necessary, its curvature fluctuation remains in the chart by
continuity at $\tau=0$.  Its coordinate $\bar v$ solves the projected equation with initial value $v_0$.
Proposition~\ref{prop:curv-LWP} gives $\bar v=v$, hence $\bar u=u$.  The mean phase
\[
\bar\vartheta(\tau):=
\int_0^1(\bar\theta(s,\tau)-\kappa s)\,ds
\]
satisfies $\bar\vartheta'=\Omega[u]$ and $\bar\vartheta(0)=\vartheta_0$, so $\bar\vartheta=\vartheta$.  Since both
representatives are centred, $\bar\eta=\eta$.

The local continuous-dependence statement is obtained by taking the same smooth reference $u_\ast$ and the same
chart for all initial curves in a sufficiently small $W^{2,2}$ neighbourhood of $\eta_0$.  The coordinate map from
curves to $(\vartheta_0,v_0)$ is continuous, and locally Lipschitz after fixing the lift of the tangent angle.
The estimates \eqref{eq:L2-continuous-dependence-v} and \eqref{eq:smooth-continuous-dependence-v}, the Sobolev
estimates for $\Phi_\ast$, the reconstruction formulae, and the phase estimate
\[
\sup_{0\le\tau\le T}|\vartheta^1(\tau)-\vartheta^2(\tau)|
\le
C_{T,\mathcal U}\bigl(|\vartheta^1_0-\vartheta^2_0|+\|v^1_0-v^2_0\|_{L^2}\bigr)
\]
give the stated continuity in $W^{2,2}$ and the positive-time $C^j$ estimate.

If $\eta_0^j\to\eta_0$ is any smooth closed approximating sequence with the same turning number, then for $j$ large
the corresponding curvatures lie in the same chart and their coordinates converge to $v_0$ in $L^2$, while their
initial phases converge to $\vartheta_0$.  The smooth gauge-fixed flows therefore converge to the solution
constructed above in
\[
C([0,T];W^{2,2})\cap C^\infty_{\mathrm{loc}}((0,T]\times\R/\Z),
\]
by the coordinate and phase estimates.  This also shows that the canonical solution is independent of the chosen
smooth reference curvature $u_\ast$.
\end{proof}

\begin{corollary}[Uniform local theory near the circle]\label{cor:uniform-circle-W22-LWP}
Fix $m\ge1$ and $\omega\in\Z\setminus\{0\}$.  There exist
$\varepsilon_{m,\omega}^{\mathrm{wp}}>0,
\qquad
T_{m,\omega}^{\mathrm{wp}}>0$
with the following property.  If $\eta_0$ is a unit-length centred $W^{2,2}$ immersion with turning number $\omega$
and
\[
\Kosc[\eta_0]<\varepsilon_{m,\omega}^{\mathrm{wp}},
\]
then the canonical solution of Theorem~\ref{thm:W22-LWP} exists at least on $[0,T_{m,\omega}^{\mathrm{wp}}]$.
The time is uniform for all such data, and the corresponding local flow map is continuous in the $W^{2,2}$ topology.
\end{corollary}

\begin{proof}
Apply Lemma~\ref{lem:closure-chart} at the round $\omega$-circle, i.e. at $u_\ast=0$.  For this reference one may
choose
\[
E_\omega:=\operatorname{span}_{\R}\{\cos(2\pi\omega s),\sin(2\pi\omega s)\},
\qquad
X_\omega:=E_\omega^\perp\cap L^2_0,
\]
because the derivative of the closure map in these two directions is an isomorphism.  If
$\Kosc[\eta_0]=\|u_0\|_2^2$ is sufficiently small, then $u_0$ belongs to this fixed chart and its coordinate has
$L^2$ norm bounded by a fixed radius $R<\rho_\omega$.  Proposition~\ref{prop:curv-LWP} gives a time depending only
on $m$, $\omega$, and $R$.  Decreasing the threshold if necessary gives the stated constants.
\end{proof}

\begin{proof}[Proof of Theorem~\ref{T:W22-LWP-main}]
This is Theorem~\ref{thm:W22-LWP}.  The formulation in the introduction is obtained by fixing the centred
unit-length representative of the initial curve.
\end{proof}


\section{The length-normalised barycentred flow and the small-\texorpdfstring{$\Kosc$}{Kosc} basin}
\label{S:len-normalised}

This section contains the analysis of the barycentred length-normalised flow introduced in
Section~\ref{S:existence}.  The main point is the invariant basin of attraction defined by small curvature
oscillation.  The evolution equations and Sobolev estimates needed for the argument are stated here, with the
perturbative remainder estimates for the small-$\Kosc$ basin proved below in this section.  The remaining routine
evolution and interpolation proofs are deferred to Appendices~\ref{A:normalised-evolution} and~\ref{A:normalised-apriori}.

Throughout this section $\eta$ denotes a smooth solution of the gauge-fixed length-normalised flow
\eqref{eq:NGFm}.  Thus
\begin{equation}\label{eq:NGFm-section}
\partial_\tau\eta
=
F\,N+\beta_F\,T-b_F,
\qquad
F=\K_m+\Lambda h,
\qquad
h:=\eta\cdot N,
\end{equation}
where
\begin{equation}\label{eq:beta-b-section}
(\beta_F)_s=kF,
\qquad
\int_\eta \beta_F\,ds=0,
\qquad
b_F:=\int_\eta(FN+\beta_FT)\,ds .
\end{equation}
The curve has unit length and zero arclength barycentre:
\begin{equation}\label{eq:unit-centred-section}
L[\eta(\cdot,\tau)]\equiv1,
\qquad
\int_{\eta(\tau)}\eta\,ds=0.
\end{equation}
We write
\[
q:=\eta\cdot T,
\qquad
D:=\partial_s .
\]
The functions $h$ and $q$ satisfy
\begin{equation}\label{eq:hq-identities-section}
h_s=-kq,
\qquad
q_s=1+kh,
\qquad
h_{ss}+k^2h=-k-qk_s .
\end{equation}

The tangential velocity fixes the constant-speed parametrisation.  For scalar intrinsic quantities we use the
transport-free derivative
\begin{equation}\label{eq:intrinsic-derivative-section}
\nabla_\tau u:=\partial_\tau u-\beta_Fu_s .
\end{equation}
The ambient translation $-b_F$ is necessary for the barycentric gauge, but it has no effect on intrinsic quantities.

\subsection{The barycentred normalisation}

\begin{proposition}[Gauge and scaling identities]\label{prop:barycentred-gauge-identities}
Along \eqref{eq:NGFm-section},
\begin{equation}\label{eq:ds-fixed-main}
\partial_\tau(ds)=0,
\qquad
\frac{\dd}{\dd \tau}\int_\eta\eta\,ds=0.
\end{equation}
Moreover,
\begin{equation}\label{eq:Lambda-main}
\Lambda(\tau)=\int_{\eta(\tau)}k\,\K_m\,ds.
\end{equation}

Conversely, let $\gamma$ be a smooth solution of the unnormalised geometric flow
\[
\partial_t^\perp\gamma=\K_m^\gamma N^\gamma
\]
with constant-speed parametrisation.  Define
\[
L_\gamma(t):=L[\gamma(\cdot,t)],
\qquad
\rho(t):=L_\gamma(t)^{-1},
\qquad
\frac{d\tau}{dt}=\rho(t)^{2m+4},
\]
and let
\[
\bar\gamma(t):=\frac1{L_\gamma(t)}\int_{\gamma(t)}\gamma\,ds_\gamma,
\qquad
\eta(u,\tau):=\rho(t)\bigl(\gamma(u,t)-\bar\gamma(t)\bigr).
\]
Then, up to a time-dependent constant shift of the parameter on $\R/\Z$, $\eta$ satisfies
\eqref{eq:NGFm-section}.  Finally,
\begin{equation}\label{eq:log-length-main}
\frac{\dd}{\dd \tau}\log \rho(\tau)=\Lambda(\tau),
\qquad
\frac{\dd}{\dd \tau}\log L_\gamma(t(\tau))=-\Lambda(\tau),
\end{equation}
and hence
\begin{equation}\label{eq:length-reconstruct-main}
L_\gamma(t(\tau))
=
L_\gamma(0)\exp\left(-\int_0^\tau\Lambda(\sigma)\,d\sigma\right).
\end{equation}
\end{proposition}

\begin{proof}
See Appendix~\ref{A:normalised-evolution}.
\end{proof}

\subsection{Evolution identities}

We use the polynomial notation $P_a^b(k)$: this denotes a finite linear combination of monomials
\[
\prod_{i=1}^a D^{j_i}k,
\qquad
j_1+\cdots+j_a=b,
\qquad
j_i\ge0,
\]
with universal coefficients depending only on $a,b$ and $m$.

\begin{proposition}[Intrinsic evolution identities]\label{prop:normalised-evolution-identities}
Along the length-normalised barycentred flow, the curvature satisfies
\begin{equation}\label{eq:k-evol-main}
\nabla_\tau k
=
(\K_m)_{ss}+k^2\K_m-\Lambda k-\Lambda qk_s .
\end{equation}
The Euler--Lagrange operator has the structure
\begin{equation}\label{eq:Km-structure-main}
\K_m
=
(-1)^{m+1}D^{2m+2}k+P_3^{2m}(k),
\end{equation}
and therefore
\begin{equation}\label{eq:Km-derivative-main}
D^\ell\K_m
=
(-1)^{m+1}D^{2m+2+\ell}k+P_3^{2m+\ell}(k)
\qquad(\ell\ge0).
\end{equation}

For
\[
H_p(\tau):=\int_{\eta(\tau)}k_{s^p}^2\,ds
\qquad(p\in\N_0),
\]
one has
\begin{align}
\label{eq:Hp-evol-main}
\frac{\dd}{\dd \tau}H_p
&=
-2\int_\eta k_{s^{p+m+2}}^2\,ds
+
2\int_\eta k_{s^p}\,P_3^{p+2m+2}(k)\,ds
+
2\int_\eta k_{s^p}\,P_5^{p+2m}(k)\,ds
\nonumber\\
&\qquad
-
\int_\eta k\,\K_m\,k_{s^p}^2\,ds
-
(2p+1)\Lambda(\tau)H_p .
\end{align}

The higher-order energy satisfies
\begin{equation}\label{eq:Em-evol-main}
\frac{\dd}{\dd \tau}E_m[\eta(\tau)]
=
-\int_\eta \K_m^2\,ds
-(2m+1)\Lambda(\tau)E_m[\eta(\tau)].
\end{equation}
Finally,
\begin{equation}\label{eq:Lambda-structure-main}
\Lambda
=
\int_\eta k_{s^{m+1}}^2\,ds
+
\int_\eta P_4^{2m}(k)\,ds .
\end{equation}
\end{proposition}

\begin{proof}
See Appendix~\ref{A:normalised-evolution}.
\end{proof}

\subsection{A priori Sobolev control}

The detailed interpolation argument is deferred to Appendix~\ref{A:normalised-apriori}.  The form used in the
main stability proof is the following.

\begin{proposition}[Sobolev differential inequalities]\label{prop:normalised-Sobolev-control}
Fix $m\ge1$ and $p\in\N_0$.  For every $\delta>0$ there is a constant
$C_{\delta}=C_{\delta}(m,p,\omega)$ and a polynomial $\mathcal P_{m,p}$ with nonnegative coefficients such that
every smooth length-normalised solution satisfies
\begin{equation}\label{eq:Sobolev-diff-main}
\frac{\dd}{\dd \tau}H_p
+
(2-\delta)\int_\eta k_{s^{p+m+2}}^2\,ds
\le
C_{\delta}\,
\mathcal P_{m,p}\bigl(1,H_0,\dots,H_{p+m+1}\bigr).
\end{equation}
More precisely, each non-leading term in \eqref{eq:Hp-evol-main} is bounded by
\[
\delta\int_\eta k_{s^{p+m+2}}^2\,ds
+
C_{\delta}\,
\mathcal P_{m,p}\bigl(1,H_0,\dots,H_{p+m+1}\bigr).
\]
\end{proposition}

\begin{proof}
See Appendix~\ref{A:normalised-apriori}.
\end{proof}

\subsection{The small higher-order-energy basin}

We first record the standard consequence of the gradient inequality from Section~\ref{S:grad-ineq}.  On a
unit-length curve, the scale-invariant energy is simply $E_m$.

\begin{theorem}[Small higher-order-energy basin]\label{thm:eta-converges-circle}
Fix $m\ge1$ and $\omega\in\Z\setminus\{0\}$.  There exist constants
$\varepsilon^E_{m,\omega}>0,
\qquad
c^E_{m,\omega}>0,
\qquad
C^E_{m,\omega}<\infty$
with the following property.  Let $\eta$ be a smooth length-normalised barycentred solution with turning number $\omega$ and
\[
E_m[\eta(0)]\le \varepsilon^E_{m,\omega}.
\]
Then the length-normalised flow exists for all $\tau\ge0$ and
\begin{equation}\label{eq:Em-exp-decay-main}
E_m[\eta(\tau)]
\le
C^E_{m,\omega}E_m[\eta(0)]e^{-c^E_{m,\omega}\tau}.
\end{equation}
Moreover, for every $r\in\N_0$ there are constants $C_r,c_r>0$ such that
\begin{equation}\label{eq:normalised-speed-decay-main}
\|\partial_\tau\eta(\cdot,\tau)\|_{C^r}
+
\|\K_m(\cdot,\tau)\|_{C^r}
+
|\Lambda(\tau)|
\le
C_r e^{-c_r\tau}.
\end{equation}
Consequently, $\eta(\cdot,\tau)\to\eta_\infty \qquad\text{in }C^\infty(\R/\Z),$
where $\eta_\infty$ is the unit-length $\omega$-circle in the barycentric gauge.  If $\eta$ is obtained from an unnormalised solution $\gamma$, then
$L_\gamma(t(\tau))\to L_\infty\in(0,\infty),$ and, after translating by the arclength barycentre,
\[
\frac{\gamma(\cdot,t)-\bar\gamma(t)}{L_\gamma(t)}
\to
\eta_\infty
\qquad\text{in }C^\infty .
\]
In particular, the unnormalised curve converges, modulo translations and reparametrisation, to an $\omega$-circle of
length $L_\infty$.
\end{theorem}

\begin{proof}
From \eqref{eq:Em-evol-main} and using the Cauchy--Schwarz inequality,
\[
|\Lambda|=\left|\int_\eta k\K_m\,ds\right|
\le
\|k\|_2\|\K_m\|_2.
\]
Therefore, for any fixed $\delta\in(0,1)$,
\[
\frac{\dd}{\dd \tau}E_m
\le
-(1-\delta)\|\K_m\|_2^2
+
C_{\delta,m}E_m^2\|k\|_2^2 .
\]
If $E_m$ is below the small-energy threshold in Theorem~\ref{thm:grad-ineq}, then
\[
\|\K_m\|_2^2\ge c_{m,\omega}E_m .
\]
Since $L\equiv1$ and $\bar k=2\pi\omega$, the PSW inequality gives
\[
\|k\|_2^2
\le
C_{m,\omega}(1+E_m).
\]
Thus, while $E_m$ is small,
\[
\frac{\dd}{\dd \tau}E_m
\le
-c_0E_m+C_0E_m^2 .
\]
Choosing $\varepsilon^E_{m,\omega}$ sufficiently small makes the small-energy condition invariant and gives
\eqref{eq:Em-exp-decay-main} by comparison with the logistic ODE.

The Sobolev differential inequalities in Proposition~\ref{prop:normalised-Sobolev-control}, together with the
exponential decay of $E_m$, give uniform bounds for all curvature derivatives.  These bounds, with the unit-length
normalisation, rule out finite-time breakdown by the continuation criterion in the fixed gauge.  Interpolation then
gives exponential decay of the curvature derivatives, and the gauge equations \eqref{eq:beta-b-section} give
exponential decay of $\K_m$, $\Lambda$, $\beta_F$, and $b_F$ in every $C^r$ norm.  This proves
\eqref{eq:normalised-speed-decay-main}.  Hence, for every $r\in\N_0$,
\[
\int_0^\infty\|\partial_\tau\eta(\cdot,\sigma)\|_{C^r}\,d\sigma<\infty .
\]
Consequently, if $\tau_2>\tau_1$, then
\[
\|\eta(\cdot,\tau_2)-\eta(\cdot,\tau_1)\|_{C^r}
\le
\int_{\tau_1}^{\tau_2}\|\partial_\tau\eta(\cdot,\sigma)\|_{C^r}\,d\sigma
\to0
\qquad(\tau_1\to\infty).
\]
Thus the whole tail of the trajectory is Cauchy in every $C^r$ norm.  In particular, every sequence
$\tau_j\to\infty$ has the same $C^r$ limit; full $C^\infty$ convergence to a smooth limit $\eta_\infty$ follows.
Since $E_m[\eta(\tau)]\to0$, the limit satisfies $k_{s^m}\equiv0$, so $k$ is constant.  The turning number and
unit-length normalisation give $k\equiv2\pi\omega$, hence $\eta_\infty$ is the unit-length $\omega$-circle.

Finally, \eqref{eq:normalised-speed-decay-main} gives $\Lambda\in L^1([0,\infty))$.  The length formula
\eqref{eq:length-reconstruct-main} then implies convergence of $L_\gamma(t(\tau))$ to a finite positive limit.
The reconstruction formula
\[
\gamma(\cdot,t)
=
\bar\gamma(t)+L_\gamma(t)\eta(\cdot,\tau(t))
\]
gives the claimed convergence of the unnormalised flow.
\end{proof}

\subsection{Curvature oscillation and the resonant modes}

We now pass from small $E_m$ to small curvature oscillation.  Since $L[\eta]\equiv1$, set
\[
\kappa:=2\pi\omega,
\qquad
f:=k-\kappa,
\qquad
e(\tau):=\int_\eta f^2\,ds=\Kosc[\eta(\tau)].
\]
Then $\int_\eta f\,ds=0$.  Define the curvature-level quadratic form
\begin{equation}\label{eq:Dm-curv-form-main}
\mathfrak D_m[f]
:=
\int_\eta\left(D^m(D^2+\kappa^2)f\right)^2\,ds .
\end{equation}

\begin{lemma}[Closure controls the resonant curvature modes]\label{lem:resonant-Kosc}
Let $\eta$ be a smooth closed unit-length curve with turning number $\omega\ne0$.  Write
\[
f(s)=k(s)-2\pi\omega
=
\sum_{n\in\Z}a_ne^{2\pi ins},
\qquad
a_0=0.
\]
Then
\begin{equation}\label{eq:resonant-Kosc-bound-main}
|a_\omega|+|a_{-\omega}|
\le
C_\omega e,
\qquad
e:=\int_0^1 f^2\,ds .
\end{equation}
\end{lemma}

\begin{proof}
Let $\psi$ be the mean-zero primitive of $f$:
\[
\psi_s=f,
\qquad
\int_0^1\psi\,ds=0.
\]
After fixing the additive constant in the tangent angle,
\[
\theta(s)=2\pi\omega s+\psi(s),
\qquad
\eta_s=e^{i\theta(s)}.
\]
Closedness gives
\[
0=\int_0^1e^{2\pi i\omega s}e^{i\psi(s)}\,ds.
\]
Since $\omega\ne0$,
\[
\int_0^1e^{2\pi i\omega s}\,ds=0,
\]
and therefore
\[
i\int_0^1e^{2\pi i\omega s}\psi(s)\,ds
=
-\int_0^1e^{2\pi i\omega s}\bigl(e^{i\psi(s)}-1-i\psi(s)\bigr)\,ds .
\]
Using $|e^{ix}-1-ix|\le \frac12x^2$,
\[
\left|\int_0^1e^{2\pi i\omega s}\psi(s)\,ds\right|
\le
\frac12\|\psi\|_2^2.
\]
Since
\[
\psi(s)=\sum_{n\ne0}\frac{a_n}{2\pi in}e^{2\pi ins},
\]
the left-hand side contains $a_{-\omega}/(2\pi i(-\omega))$.  Hence
\[
|a_{-\omega}|
\le
\pi|\omega|\,\|\psi\|_2^2
\le
C_\omega\|f\|_2^2
=
C_\omega e,
\]
where the last inequality is the PSW inequality for the mean-zero primitive.  Since $f$ is real,
$a_\omega=\overline{a_{-\omega}}$.
\end{proof}

\begin{lemma}[Curvature spectral gap modulo closure]\label{lem:curv-spectral-gap-Kosc}
Let
\begin{equation}\label{eq:mu-m-omega-curv-main}
\mu_{m,\omega}
:=
(2\pi)^{2m+4}
\min_{n\in\Z\setminus\{0,\pm\omega\}}
|n|^{2m}(n^2-\omega^2)^2
>0 .
\end{equation}
Then
\begin{equation}\label{eq:Dm-controls-e-main}
\mathfrak D_m[f]
\ge
\mu_{m,\omega}e-C_{m,\omega}e^2.
\end{equation}
Moreover, for each $0\le r\le m+2$,
\begin{equation}\label{eq:Dm-controls-derivatives-main}
\|D^rf\|_2^2
\le
C_{m,r,\omega}\bigl(\mathfrak D_m[f]+e^2\bigr).
\end{equation}
\end{lemma}

\begin{proof}
For the Fourier mode $e^{2\pi ins}$,
\[
D^m(D^2+\kappa^2)e^{2\pi ins}
=
(2\pi in)^m\bigl((2\pi\omega)^2-(2\pi n)^2\bigr)e^{2\pi ins}.
\]
Thus
\[
\mathfrak D_m[f]
=
\sum_{n\ne0}
(2\pi)^{2m+4}|n|^{2m}(n^2-\omega^2)^2|a_n|^2 .
\]
The modes $n=\pm\omega$ are the translation modes and do not contribute to $\mathfrak D_m$.  Lemma~\ref{lem:resonant-Kosc}
gives
\[
\sum_{n=\pm\omega}|a_n|^2\le C_\omega e^2.
\]
Therefore
\[
\mathfrak D_m[f]
\ge
\mu_{m,\omega}
\sum_{n\in\Z\setminus\{0,\pm\omega\}}|a_n|^2
\ge
\mu_{m,\omega}e-C_{m,\omega}e^2 .
\]
The derivative estimate follows similarly, because for $n\notin\{0,\pm\omega\}$ and $0\le r\le m+2$,
\[
(2\pi|n|)^{2r}
\le
C_{m,r,\omega}
(2\pi)^{2m+4}|n|^{2m}(n^2-\omega^2)^2,
\]
while the resonant contribution is again bounded by $C_{m,r,\omega}e^2$.
\end{proof}

\begin{lemma}[Quadratic expansion of curvature oscillation]\label{lem:Kosc-quadratic-expansion}
Along the length-normalised barycentred flow,
\begin{equation}\label{eq:eprime-quadratic-form-main}
\frac{\dd}{\dd \tau}e
=
-2\mathfrak D_m[f]+\mathcal R_m[f],
\end{equation}
where there are constants $\theta_m\in(0,1)$ and $C_{m,\omega}<\infty$ such that, whenever $e\le1$,
\begin{equation}\label{eq:Kosc-remainder-bound-main}
|\mathcal R_m[f]|
\le
C_{m,\omega}e^{\theta_m}\mathfrak D_m[f]
+
C_{m,\omega}e^{1+\theta_m}.
\end{equation}
\end{lemma}

\begin{proof}
Since $L\equiv1$ and $\bar k=\kappa=2\pi\omega$ is fixed,
\[
e=\int_\eta(k-\kappa)^2\,ds=\int_\eta k^2\,ds-\kappa^2.
\]
For a normal speed $F$, the intrinsic identity gives
\[
\frac{\dd}{\dd \tau}\int_\eta k^2\,ds
=
\int_\eta(2k_{ss}+k^3)F\,ds.
\]
Set
\[
A:=2k_{ss}+k^3,
\qquad
H_0:=\int_\eta k^2\,ds=\kappa^2+e.
\]
For $F=\K_m+\Lambda h$ and $\Lambda=\int_\eta k\K_m\,ds$, the scaling identity for
$E_0=\frac12\int k^2\,ds$ gives
\[
\int_\eta Ah\,ds=-H_0.
\]
Hence
\[
e'
=
\int_\eta A\K_m\,ds-\Lambda H_0
=
\int_\eta(A-H_0k)\K_m\,ds .
\]
Write
\[
\Phi:=(D^2+\kappa^2)f,
\qquad
B:=3\kappa f^2+f^3-e(\kappa+f).
\]
Since $k=\kappa+f$, a direct expansion gives
\[
A-H_0k=2\Phi+B .
\]
The formula \eqref{E:Km} gives the exact decomposition
\[
\K_m=\mathcal L_m f+\mathcal N_m^{(2)}[f]+\mathcal N_m^{(3)}[f],
\qquad
\mathcal L_m f:=(-1)^{m+1}D^{2m}\Phi,
\]
where, with the sums absent when $m=1$,
\begin{align*}
\mathcal N_m^{(2)}[f]
&:=
2(-1)^{m+1}\kappa fD^{2m}f
-\frac{\kappa}{2}(D^mf)^2
+\kappa\sum_{j=1}^{m-1}(-1)^{j+1}D^{m-j}f\,D^{m+j}f,\\
\mathcal N_m^{(3)}[f]
&:=
(-1)^{m+1}f^2D^{2m}f
-\frac12 f(D^mf)^2
+f\sum_{j=1}^{m-1}(-1)^{j+1}D^{m-j}f\,D^{m+j}f.
\end{align*}
Thus the quadratic part of $e'$ is
\[
2\int_\eta\Phi\,\mathcal L_mf\,ds
=
2(-1)^{m+1}\int_\eta\Phi\,D^{2m}\Phi\,ds
=
-2\int_\eta(D^m\Phi)^2\,ds
=
-2\mathfrak D_m[f].
\]
Consequently
\begin{align}
\label{E:Rm}
\mathcal R_m[f] 
&=
2\int_\eta\Phi\,\mathcal N_m^{(2)}\,ds
+2\int_\eta\Phi\,\mathcal N_m^{(3)}\,ds +\int_\eta B_2\,\mathcal L_mf\,ds
+\int_\eta B_3\,\mathcal L_mf\,ds\\ &\quad
+\int_\eta B_e\,\mathcal L_mf\,ds 
+\int_\eta(B_2+B_3)(\mathcal N_m^{(2)}+\mathcal N_m^{(3)})\,ds
+\int_\eta B_e(\mathcal N_m^{(2)}+\mathcal N_m^{(3)})\,ds, \nonumber
\end{align}
where
\[
B_2:=3\kappa f^2,
\qquad
B_3:=f^3,
\qquad
B_e:=-e(\kappa+f).
\]
It remains to estimate these displayed groups one by one.

We first record the interpolation consequence used below.  Put
\[
R:=m+2,
\qquad
X:=\|D^Rf\|_2.
\]
By the standard one-dimensional periodic Gagliardo--Nirenberg inequality, H\"older's inequality, and the spectral estimate
\eqref{eq:Dm-controls-derivatives-main}, there is $\theta_m\in(0,1)$ such that every monomial integral occurring
in one of the following three classes 
\[
\begin{array}{ll}
\text{(i)} & a=3\text{ or }4,\qquad \sum_i j_i\le 2R-2,\\[2mm]
\text{(ii)}& a=4,5,\text{ or }6,\qquad \sum_i j_i\le 2R-4,\\[2mm]
\text{(iii)}& \text{an explicit extra factor }e\text{ times an integral with }\\
& a=2,3,4\text{ and }\sum_i j_i\le2R-4.
\end{array}
\]
satisfies
\begin{equation}\label{eq:Kosc-proof-monomial-bound}
\int_\eta\prod_{i=1}^a |D^{j_i}f|\,ds
\le
C_{m,\omega}e^{\theta_m}\bigl(\mathfrak D_m[f]+e^2\bigr)
+
C_{m,\omega}e^{1+\theta_m}.
\end{equation}

Indeed, after integrating by parts until no factor has more than $R$ derivatives, placing any $D^Rf$ factor in $L^2$ directly and using H\"older together with
Lemma~\ref{lem:appendix-GN} for the lower-order factors gives a finite sum of terms
\[
C X^\beta e^{(a-\beta)/2},
\qquad
0\le\beta\le\frac{\sum_i j_i+a/2-1}{R-1/2},
\]
with lower-order Gagliardo--Nirenberg contributions corresponding to smaller values of $\beta$ and larger powers
of $e$.  In class (i), either $\beta<2$, or $a=4$, $\sum_i j_i=2R-2$, and the borderline term is $CeX^2$.
In class (ii) one has $\beta<2$.  In class (iii) the explicit factor $e$ supplies the same small coefficient.
Young's inequality, with $\theta_m$ chosen sufficiently small depending only on $m$, gives
\eqref{eq:Kosc-proof-monomial-bound}; finally
$X^2\le C_{m,\omega}(\mathfrak D_m[f]+e^2)$ by \eqref{eq:Dm-controls-derivatives-main}.

We now estimate the seven groups in \eqref{E:Rm}.

First, $\mathcal N_m^{(2)}$ consists of the quadratic monomials
\[
fD^{2m}f,
\qquad
(D^mf)^2,
\qquad
D^{m-j}f\,D^{m+j}f\quad(1\le j\le m-1),
\]
each of total derivative order at most $2m=2R-4$.  Multiplication by
$\Phi=D^2f+\kappa^2f$ gives three-factor monomials of total derivative order at most
$2m+2=2R-2$.  After integrating by parts if necessary to lower derivatives above $R$,
\eqref{eq:Kosc-proof-monomial-bound}, class (i), gives
\begin{equation}\label{eq:Kosc-proof-Phi-N2}
\left|2\int_\eta\Phi\,\mathcal N_m^{(2)}\,ds\right|
\le
C_{m,\omega}e^{\theta_m}\bigl(\mathfrak D_m[f]+e^2\bigr)
+
C_{m,\omega}e^{1+\theta_m}.
\end{equation}

Second, $\mathcal N_m^{(3)}$ consists of the cubic monomials
\[
f^2D^{2m}f,
\qquad
f(D^mf)^2,
\qquad
fD^{m-j}f\,D^{m+j}f\quad(1\le j\le m-1),
\]
again of total derivative order at most $2m=2R-4$.  Multiplication by $\Phi$ gives four-factor monomials of total
order at most $2R-2$.  Hence class (i) gives
\begin{equation}\label{eq:Kosc-proof-Phi-N3}
\left|2\int_\eta\Phi\,\mathcal N_m^{(3)}\,ds\right|
\le
C_{m,\omega}e^{\theta_m}\bigl(\mathfrak D_m[f]+e^2\bigr)
+
C_{m,\omega}e^{1+\theta_m}.
\end{equation}

Third, consider the term with $B_2=3\kappa f^2$.  Since
$\mathcal L_mf=(-1)^{m+1}D^{2m}\Phi$, periodic integration by parts gives
\[
\int_\eta B_2\mathcal L_mf\,ds
=
C_\kappa\int_\eta D^m(f^2)D^m\Phi\,ds.
\]
The Leibniz expansion of $D^m(f^2)$ gives terms
\[
D^af\,D^{m-a}f\,D^m\Phi,
\qquad
0\le a\le m,
\]
and $D^m\Phi=D^{m+2}f+\kappa^2D^mf$.  Thus every resulting monomial has three factors and total derivative
order at most $m+(m+2)=2R-2$.  Class (i) yields
\begin{equation}\label{eq:Kosc-proof-B2-L}
\left|\int_\eta B_2\mathcal L_mf\,ds\right|
\le
C_{m,\omega}e^{\theta_m}\bigl(\mathfrak D_m[f]+e^2\bigr)
+
C_{m,\omega}e^{1+\theta_m}.
\end{equation}

Fourth, the term with $B_3=f^3$ is identical except that $D^m(f^3)$ is a sum of products of three derivatives of
$f$ whose derivative orders add to $m$.  Multiplication by $D^m\Phi$ gives four-factor monomials of total order at
most $2R-2$, and therefore
\begin{equation}\label{eq:Kosc-proof-B3-L}
\left|\int_\eta B_3\mathcal L_mf\,ds\right|
\le
C_{m,\omega}e^{\theta_m}\bigl(\mathfrak D_m[f]+e^2\bigr)
+
C_{m,\omega}e^{1+\theta_m}.
\end{equation}

Fifth, the explicit-$e$ part of $B$ is lower order.  Since
\[
\int_\eta\mathcal L_mf\,ds=0,
\]
the constant part $-\kappa e$ of $B_e$ does not contribute.  The remaining part satisfies
\begin{align*}
\left|\int_\eta (-ef)\mathcal L_mf\,ds\right|
&\le
e\left(\|D^{m+1}f\|_2^2+\kappa^2\|D^mf\|_2^2\right)\\
&\le
C_{m,\omega}e\bigl(\mathfrak D_m[f]+e^2\bigr)
\le
C_{m,\omega}e^{\theta_m}\mathfrak D_m[f]+C_{m,\omega}e^{1+\theta_m},
\end{align*}
where we used \eqref{eq:Dm-controls-derivatives-main} and $e\le1$.

Sixth, multiply the non-explicit part $B_2+B_3$ by the nonlinear curvature terms.  The product
$(B_2+B_3)\mathcal N_m^{(2)}$ contains four- and five-factor monomials of total derivative order at most
$2m=2R-4$, while $(B_2+B_3)\mathcal N_m^{(3)}$ contains five- and six-factor monomials of the same total order.
Class (ii) gives
\begin{equation}\label{eq:Kosc-proof-B23-N}
\left|\int_\eta(B_2+B_3)(\mathcal N_m^{(2)}+\mathcal N_m^{(3)})\,ds\right|
\le
C_{m,\omega}e^{\theta_m}\bigl(\mathfrak D_m[f]+e^2\bigr)
+
C_{m,\omega}e^{1+\theta_m}.
\end{equation}

Seventh, consider the remaining explicit-$e$ nonlinear term.  The factor $-\kappa e$ multiplying
$\mathcal N_m^{(2)}$ gives $e$ times quadratic monomials of total derivative order at most $2R-4$; multiplying
$\mathcal N_m^{(3)}$ gives $e$ times cubic monomials of total derivative order at most $2R-4$.  The factor
$-ef$ adds one undifferentiated $f$ and gives $e$ times three- or four-factor monomials of the same total
derivative order.  Class (iii) gives
\begin{equation}\label{eq:Kosc-proof-Be-N}
\left|\int_\eta B_e(\mathcal N_m^{(2)}+\mathcal N_m^{(3)})\,ds\right|
\le
C_{m,\omega}e^{\theta_m}\bigl(\mathfrak D_m[f]+e^2\bigr)
+
C_{m,\omega}e^{1+\theta_m}.
\end{equation}

Combining \eqref{eq:Kosc-proof-Phi-N2}--\eqref{eq:Kosc-proof-Be-N} gives
\[
|\mathcal R_m[f]|
\le
C_{m,\omega}e^{\theta_m}\bigl(\mathfrak D_m[f]+e^2\bigr)
+
C_{m,\omega}e^{1+\theta_m}.
\]
When $e\le1$, the term $e^{\theta_m}e^2$ is bounded by $e^{1+\theta_m}$.  Therefore
\[
|\mathcal R_m[f]|
\le
C_{m,\omega}e^{\theta_m}\mathfrak D_m[f]
+
C_{m,\omega}e^{1+\theta_m},
\]
which proves \eqref{eq:Kosc-remainder-bound-main} and hence the lemma.
\end{proof}

\begin{proposition}[Invariant small-oscillation basin]\label{prop:Kosc-decay}
There exist constants
$\varepsilon^{\mathrm{osc}}_{m,\omega}>0,
\qquad
c^{\mathrm{osc}}_{m,\omega}>0,
\qquad
C^{\mathrm{osc}}_{m,\omega}<\infty$
such that if a smooth length-normalised solution satisfies
\[
e(0)=\Kosc[\eta(0)]\le\varepsilon^{\mathrm{osc}}_{m,\omega},
\]
then, on its interval of existence,
\begin{equation}\label{eq:Kosc-exp-decay-main}
e(\tau)\le
2e(0)e^{-c^{\mathrm{osc}}_{m,\omega}\tau}.
\end{equation}
Moreover,
\begin{equation}\label{eq:integrated-Dm-bound-main}
\int_\tau^{\tau+1}\mathfrak D_m[f(\sigma)]\,d\sigma
\le
C^{\mathrm{osc}}_{m,\omega}e(\tau)
\end{equation}
whenever the interval $[\tau,\tau+1]$ lies in the existence interval.
\end{proposition}

\begin{proof}
Combining Lemma~\ref{lem:Kosc-quadratic-expansion} with Lemma~\ref{lem:curv-spectral-gap-Kosc} gives, whenever
$e\le1$,
\[
e'
\le
-2\mathfrak D_m[f]
+
C_{m,\omega}e^{\theta_m}\mathfrak D_m[f]
+
C_{m,\omega}e^{1+\theta_m}.
\]
Choose $\varepsilon^{\mathrm{osc}}_{m,\omega}$ so small that
$C_{m,\omega}e^{\theta_m}\le1$ whenever $e\le2\varepsilon^{\mathrm{osc}}_{m,\omega}$.  Then
\[
e'
\le
-\mathfrak D_m[f]+C_{m,\omega}e^{1+\theta_m}.
\]
Using \eqref{eq:Dm-controls-e-main} and $e^2\le e^{1+\theta_m}$ for $e\le1$ gives
\[
e'
\le
-\mu_{m,\omega}e+C_{m,\omega}e^{1+\theta_m}.
\]
After decreasing $\varepsilon^{\mathrm{osc}}_{m,\omega}$ once more, the last term is bounded by
$\frac12\mu_{m,\omega}e$ whenever $e\le2\varepsilon^{\mathrm{osc}}_{m,\omega}$.  Hence
\[
e'
\le
-\frac12\mu_{m,\omega}e
\]
as long as $e\le2\varepsilon^{\mathrm{osc}}_{m,\omega}$.  A standard continuation argument proves
\eqref{eq:Kosc-exp-decay-main} with
\[
c^{\mathrm{osc}}_{m,\omega}:=\frac12\mu_{m,\omega}.
\]

For the integrated estimate, the same smallness gives
\[
\mathfrak D_m[f]
\le
-e' + C_{m,\omega}e^{1+\theta_m}.
\]
Integrating from $\tau$ to $\tau+1$ and using the exponential decay of $e$ gives
\[
\int_\tau^{\tau+1}\mathfrak D_m[f(\sigma)]\,d\sigma
\le
e(\tau)+C_{m,\omega}\int_\tau^{\tau+1}e(\sigma)^{1+\theta_m}\,d\sigma
\le
C^{\mathrm{osc}}_{m,\omega}e(\tau).
\]
\end{proof}

\begin{corollary}[Positive-time smoothing in the small-\texorpdfstring{\(\Kosc\)}{Kosc} basin]\label{cor:positive-time-smoothing-Kosc}
After possibly decreasing $\varepsilon^{\mathrm{osc}}_{m,\omega}$, every smooth length-normalised solution with
\[
\Kosc[\eta(0)]\le\varepsilon^{\mathrm{osc}}_{m,\omega}
\]
exists for all $\tau\ge0$ and satisfies, for every $p\in\N_0$ and every $\delta>0$,
\begin{equation}\label{eq:positive-time-smoothing-main}
\sup_{\tau\ge\delta}\int_{\eta(\tau)}k_{s^p}^2\,ds
+
\sup_{\tau\ge\delta}\|k_{s^p}(\cdot,\tau)\|_{L^\infty}
\le
C_{p,\delta,m,\omega,\Kosc[\eta(0)]}.
\end{equation}
The same estimates hold for canonical relaxed solutions obtained from Theorem~\ref{thm:W22-LWP}.
\end{corollary}

\begin{proof}
See Appendix~\ref{A:normalised-apriori}.
\end{proof}

\begin{corollary}[Entrance into the small-\(E_m\) basin]\label{cor:Kosc-enters-Em-basin}
Let $\varepsilon^E_{m,\omega}$ be the threshold from Theorem~\ref{thm:eta-converges-circle}.  If
\[
\Kosc[\eta(0)]\le\varepsilon^{\mathrm{osc}}_{m,\omega},
\]
with $\varepsilon^{\mathrm{osc}}_{m,\omega}$ sufficiently small, then there exists $\tau_1\ge0$ such that
\[
E_m[\eta(\tau_1)]<\varepsilon^E_{m,\omega}.
\]
\end{corollary}

\begin{proof}
By \eqref{eq:integrated-Dm-bound-main}, for every $\tau\ge0$ there is
$\sigma_\tau\in[\tau,\tau+1]$ such that
\[
\mathfrak D_m[f(\sigma_\tau)]
\le
C_{m,\omega}e(\tau).
\]
Taking $r=m$ in \eqref{eq:Dm-controls-derivatives-main},
\[
\|D^mf(\sigma_\tau)\|_2^2
\le
C_{m,\omega}\bigl(\mathfrak D_m[f(\sigma_\tau)]+e(\sigma_\tau)^2\bigr)
\le
C_{m,\omega}e(\tau).
\]
Since $D^mk=D^mf$ and
\[
E_m[\eta(\sigma_\tau)]=\frac12\|D^mk(\sigma_\tau)\|_2^2,
\]
we obtain
\[
E_m[\eta(\sigma_\tau)]
\le
C_{m,\omega}e(\tau).
\]
The right-hand side tends to zero exponentially by Proposition~\ref{prop:Kosc-decay}.  Choosing $\tau$ large gives
the desired time $\tau_1:=\sigma_\tau$.
\end{proof}

\subsection{Smooth small-\texorpdfstring{\(\Kosc\)}{Kosc} stability}

\begin{theorem}[Smooth small-\texorpdfstring{\(\Kosc\)}{Kosc} stability]\label{thm:smooth-small-Kosc-stability}
Fix $m\ge1$ and $\omega\in\Z\setminus\{0\}$.  There exists
$\varepsilon^{\mathrm{sm}}_{m,\omega}>0$ with the following property.  Let $\gamma_0$ be a smooth closed immersed curve with turning number $\omega$ and
\[
\Kosc[\gamma_0]<\varepsilon^{\mathrm{sm}}_{m,\omega}.
\]
Then the unnormalised generalised ideal flow exists for all $t\ge0$, its length remains uniformly bounded and
converges to a finite positive limit, and the length-normalised barycentred flow satisfies
\begin{equation}\label{eq:small-Kosc-decay-theorem}
\Kosc[\eta(\tau)]
\le
C_{m,\omega}\Kosc[\gamma_0]e^{-c_{m,\omega}\tau}.
\end{equation}
Moreover,
$\eta(\cdot,\tau)\to\eta_\infty
\qquad\text{in }C^\infty,$
where $\eta_\infty$ is the unit-length $\omega$-circle.  Equivalently, the unnormalised flow converges, modulo
translations and reparametrisation, to an $\omega$-circle of finite positive length.
\end{theorem}

\begin{proof}
Length-normalise the initial curve and run the barycentred flow.  Since $\Kosc$ is scale invariant, the initial
smallness is exactly the smallness of $e(0)$ for the unit-length curve.  Choose
\[
\varepsilon^{\mathrm{sm}}_{m,\omega}\le\varepsilon^{\mathrm{osc}}_{m,\omega}.
\]
Proposition~\ref{prop:Kosc-decay} gives the exponential decay of $\Kosc$ on the existence interval.  Corollary
\ref{cor:positive-time-smoothing-Kosc} prevents loss of smoothness and gives global existence in $\tau$.
Corollary~\ref{cor:Kosc-enters-Em-basin} gives a finite time $\tau_1$ at which the flow enters the small-$E_m$
basin.  Applying Theorem~\ref{thm:eta-converges-circle} from time $\tau_1$ gives smooth exponential convergence to
the unit-length $\omega$-circle.

The dilation coefficient $\Lambda$ is integrable by Theorem~\ref{thm:eta-converges-circle}.  Hence
\eqref{eq:length-reconstruct-main} shows that the original length converges to a finite positive limit.  Since this
limit is finite and positive, the relation $dt/d\tau=L_\gamma^{2m+4}$ implies $t\to\infty$ as $\tau\to\infty$.
The reconstruction formula gives the convergence of the unnormalised flow.
\end{proof}

\subsection{Bounded length and eventual circularity}

The preceding smooth basin theorem also gives the bounded-length convergence result after any positive time.

\begin{proposition}[Bounded length forces vanishing curvature oscillation]\label{prop:length-bound-Kosc}
Let $m\ge1$ and let $\gamma:\R/\Z\times[0,\infty)\to\R^2$ be a smooth unnormalised solution with turning number
$\omega$.  Assume that
\[
L[\gamma(\cdot,t)]\le L_\ast
\qquad(t\ge0).
\]
Then
\[
\Kosc[\gamma(\cdot,t)]\to0
\qquad(t\to\infty).
\]
More precisely, if $E_m[\gamma_0]>0$, then
\begin{equation}\label{eq:Kosc-rate-length-bound-main}
\Kosc[\gamma(\cdot,t)]
\le
\frac{
2(2\pi)^{-2m}L_\ast^{2m+1}
}{
E_m[\gamma_0]^{-1}+(2m+1)^2L_\ast^{-3}t
}.
\end{equation}
\end{proposition}

\begin{proof}
The dissipation identity gives
\[
\frac{d}{dt}E_m[\gamma(t)]
=
-\|\K_m(\cdot,t)\|_{L^2(\gamma(t))}^2.
\]
By Proposition~\ref{prop:weak-grad-ineq},
\[
\|\K_m(\cdot,t)\|_{L^2(\gamma(t))}^2
\ge
(2m+1)^2L(t)^{-3}E_m[\gamma(t)]^2
\ge
(2m+1)^2L_\ast^{-3}E_m[\gamma(t)]^2.
\]
Thus
\[
\frac{d}{dt}E_m[\gamma(t)]
\le
-(2m+1)^2L_\ast^{-3}E_m[\gamma(t)]^2.
\]
If $E_m[\gamma_0]=0$, then $k$ is initially constant, so the solution is already an $\omega$-circle and the claim is
trivial.  Otherwise integration gives
\[
E_m[\gamma(t)]
\le
\frac{1}{E_m[\gamma_0]^{-1}+(2m+1)^2L_\ast^{-3}t}.
\]
Finally, the PSW inequality applied to $k-\bar k$ gives
\[
\Kosc[\gamma(t)]
=
L(t)\int_{\gamma(t)}(k-\bar k)^2\,ds
\le
2(2\pi)^{-2m}L(t)^{2m+1}E_m[\gamma(t)],
\]
and \eqref{eq:Kosc-rate-length-bound-main} follows.
\end{proof}

\begin{theorem}[Bounded-length convergence]\label{thm:bounded-length-convergence}
Let $m\ge1$ and $\omega\ne0$.  Let $\gamma_0$ be a closed $W^{2,2}$ immersed curve with turning number
$\omega$, and let $\gamma$ be an immortal canonical relaxed unnormalised $m$-ideal flow with initial value
$\gamma_0$.  Assume that $\gamma$ is smooth for every positive time and that its length remains uniformly bounded:
\[
\sup_{t\ge0}L[\gamma(\cdot,t)]<\infty .
\]
Then $\gamma$ converges, modulo translations and reparametrisation, to an $\omega$-circle of finite positive length.
The convergence is exponential in the length-normalised time, and hence exponential in the original time after a
finite positive time shift.  The canonical relaxed flow is unique and is the positive-time limit of every smooth
approximating sequence preserving closedness and turning number.
\end{theorem}

\begin{proof}
Choose $t_0>0$.  The time-shifted curve $\gamma(\cdot,t_0+\cdot)$ is a smooth immortal unnormalised solution
with the same turning number and with uniformly bounded length.  Proposition~\ref{prop:length-bound-Kosc}, applied
to this time-shifted smooth flow, gives
\[
\Kosc[\gamma(\cdot,t)]\to0
\qquad(t\to\infty).
\]
Choose $t_1>t_0$ so large that
\[
\Kosc[\gamma(\cdot,t_1)]<\varepsilon^{\mathrm{sm}}_{m,\omega}.
\]
Since $\gamma(\cdot,t_1)$ is smooth, Theorem~\ref{thm:smooth-small-Kosc-stability} applies to the time-shifted
flow starting at $t_1$.  The time-shifted length converges to a finite positive limit, so
$d\tau/dt=L(t)^{-(2m+4)}$ is bounded above and below by positive constants for all sufficiently large time.
Exponential convergence in $\tau$ therefore gives exponential convergence in the original time variable after
$t_1$.  The uniqueness and approximation assertions are inherited from the canonical local $W^{2,2}$ construction
and propagated by restarting the flow on compact positive-time intervals.
\end{proof}

\begin{proof}[Proof of Theorem~\ref{T:bounded-length-convergence}]
This is Theorem~\ref{thm:bounded-length-convergence}.
\end{proof}

\subsection{Global relaxed flows from \texorpdfstring{$W^{2,2}$}{W22} data}

We finally combine the curvature-coordinate local theory from Section~\ref{S:existence} with the small-\(\Kosc\)
basin.

\begin{theorem}[Canonical global relaxed flow in the small-\texorpdfstring{\(\Kosc\)}{Kosc} basin]\label{thm:canonical-W22-global}
Fix $m\ge1$ and $\omega\in\Z\setminus\{0\}$.  There exists
$\varepsilon^{W^{2,2}}_{m,\omega}>0$ with the following property.

Let $\eta_0$ be a unit-length, centred, arclength-parametrised $W^{2,2}$ immersed closed curve with
\[
\omega[\eta_0]=\omega,
\qquad
\Kosc[\eta_0]<\varepsilon^{W^{2,2}}_{m,\omega}.
\]
Then there exists a unique canonical gauge-fixed length-normalised $m$-ideal flow
$\eta:\R/\Z\times[0,\infty)\to\R^2$
with initial value $\eta_0$, satisfying
$\eta\in C([0,\infty);W^{2,2})\cap C^\infty((0,\infty)\times\R/\Z).$
For every $\tau>0$, $\eta(\cdot,\tau)$ is a classical solution of the gauge-fixed length-normalised equation
\eqref{eq:NGFm}.  For every smooth approximating sequence $\eta_0^j\to\eta_0$ in $W^{2,2}$ preserving
closedness and turning number, the corresponding smooth gauge-fixed length-normalised flows converge to $\eta$ in
\[
C([0,T];W^{2,2})\cap C^\infty_{\mathrm{loc}}((0,T]\times\R/\Z)
\]
for every $T<\infty$.  The resulting flow map is a continuous semiflow on the small-$\Kosc$ $W^{2,2}$
neighbourhood, modulo Euclidean motions and reparametrisation.

Moreover,
\begin{equation}\label{eq:W22-Kosc-decay-main}
\Kosc[\eta(\tau)]
\le
C_{m,\omega}\Kosc[\eta_0]e^{-c_{m,\omega}\tau},
\end{equation}
and
$\eta(\cdot,\tau)\to\eta_\infty
\qquad\text{in }C^\infty$ as
$\tau\to\infty$, where $\eta_\infty$ is the unit-length $\omega$-circle in the barycentric gauge.
\end{theorem}

\begin{proof}
Let
$\varepsilon^{W^{2,2}}_{m,\omega}
\le
\min\{\varepsilon^{\mathrm{osc}}_{m,\omega},\varepsilon^{\mathrm{wp}}_{m,\omega}\},$
where $\varepsilon^{\mathrm{wp}}_{m,\omega}$ is the uniform near-circle local well-posedness threshold from
Corollary~\ref{cor:uniform-circle-W22-LWP}.  Choose smooth approximations $\eta_0^j$ from
Lemma~\ref{lem:W22-density-closure}.  For $j$ sufficiently large,
\[
\Kosc[\eta_0^j]<\varepsilon^{\mathrm{osc}}_{m,\omega},
\qquad
\Kosc[\eta_0^j]\to\Kosc[\eta_0].
\]
Let $\eta^j$ be the corresponding smooth length-normalised flows.  Proposition~\ref{prop:Kosc-decay} gives
uniformly in $j$
\begin{equation}\label{eq:W22-proof-smooth-Kosc-decay}
e_j(\tau):=\Kosc[\eta^j(\tau)]
\le
2e_j(0)e^{-c^{\mathrm{osc}}_{m,\omega}\tau},
\end{equation}
and Corollary~\ref{cor:positive-time-smoothing-Kosc} gives uniform positive-time curvature bounds.

We first pass these estimates to the canonical relaxed solution.  On the local interval supplied by
Theorem~\ref{thm:W22-LWP}, the smooth flows converge to the canonical solution in
$C([0,T];W^{2,2})$ and smoothly on compact positive-time subintervals.  Hence
\eqref{eq:W22-proof-smooth-Kosc-decay} passes to the limit there.  In particular the relaxed curvature coordinate
stays in the same small $L^2$ ball.  Since the near-circle local existence time in
Corollary~\ref{cor:uniform-circle-W22-LWP} is uniform on this ball, the curvature-coordinate solution can be
restarted at the end of the interval.  Iterating this
argument gives a global relaxed solution and the decay estimate \eqref{eq:W22-Kosc-decay-main}.  The same passage
gives the positive-time smoothing estimates for the relaxed flow.

It remains to justify entrance into the small-$E_m$ basin for the relaxed solution.  Fix $T>0$ and put
$I_T:=[T,T+1]$.  For the smooth approximating flows, \eqref{eq:Dm-controls-derivatives-main},
\eqref{eq:integrated-Dm-bound-main}, and \eqref{eq:W22-proof-smooth-Kosc-decay} give
\begin{multline}\label{eq:W22-proof-space-time-Em}
\int_{I_T}E_m[\eta^j(\sigma)]\,d\sigma
=
\frac12\int_{I_T}\|D^m(k^j-\kappa)(\sigma)\|_2^2\,d\sigma \\
\le
C_{m,\omega}
\int_{I_T}\bigl(\mathfrak D_m[f_j(\sigma)]+e_j(\sigma)^2\bigr)\,d\sigma 
\le
C_{m,\omega}e_j(T)
\le
C_{m,\omega}e_j(0)e^{-c^{\mathrm{osc}}_{m,\omega}T},
\end{multline}
where $f_j=k^j-\kappa$.  The constants are independent of $j$.

By \eqref{eq:W22-proof-space-time-Em}, the sequence $f_j$ is bounded in $L^2(I_T;H^m)$.  After passing to a
subsequence it converges weakly in $L^2(I_T;H^m)$ to some limit.  The already established
$C(I_T;L^2)$ convergence of the curvature coordinates identifies that weak limit with the curvature fluctuation
$f=k-\kappa$ of the relaxed solution.  Lower semicontinuity of the $L^2$ norm therefore yields
\begin{equation}\label{eq:W22-proof-lsc-Em}
\int_{I_T}E_m[\eta(\sigma)]\,d\sigma
\le
\liminf_{j\to\infty}
\int_{I_T}E_m[\eta^j(\sigma)]\,d\sigma
\le
C_{m,\omega}
\Kosc[\eta_0]e^{-c^{\mathrm{osc}}_{m,\omega}T}.
\end{equation}
Choose $T$ so large that the right-hand side is strictly smaller than
$\varepsilon^E_{m,\omega}$, the threshold in Theorem~\ref{thm:eta-converges-circle}.  Since $I_T$ has length
$1$ and the relaxed solution is smooth for positive time, \eqref{eq:W22-proof-lsc-Em} gives a time
$\tau_1\in I_T$ with
\[
E_m[\eta(\tau_1)]<\varepsilon^E_{m,\omega}.
\]
Applying Theorem~\ref{thm:eta-converges-circle} to the smooth solution starting at $\tau_1$ gives smooth
convergence to the unit-length $\omega$-circle.

The convergence statement for arbitrary smooth approximating sequences follows in the same way.  On the first
uniform local interval it is exactly the approximation statement in Theorem~\ref{thm:W22-LWP}; the small-$\Kosc$
decay keeps all approximating solutions in the same uniform restart neighbourhood, so iteration gives convergence
on every compact time interval.  The positive-time smoothing estimates upgrade the convergence to
$C^\infty_{\mathrm{loc}}((0,T]\times\R/\Z)$.

Continuity of the semiflow follows from the $L^2$ curvature-coordinate continuous dependence in
Theorem~\ref{thm:W22-LWP}, the uniform restart just described, and the reconstruction estimates.
\end{proof}

\begin{corollary}[Relaxed unnormalised flow from rough data]\label{cor:W22-relaxed-unnormalised}
Fix $m\ge1$ and $\omega\in\Z\setminus\{0\}$.
There exists $\varepsilon^{W^{2,2}}_{m,\omega}>0$ with the following property.
Let $\gamma_0$ be a closed $W^{2,2}$ immersed curve with turning number $\omega$ and
\[
\Kosc[\gamma_0]<\varepsilon^{W^{2,2}}_{m,\omega}.
\]
Then, after choosing a unit-length centred representative, there is a unique global canonical relaxed
length-normalised $m$-ideal flow starting from $\gamma_0$.  It is smooth for every positive time, depends
continuously on the initial curve in the $W^{2,2}$ topology, is the positive-time limit of every smooth
approximating sequence preserving closedness and turning number, and converges smoothly to the unit-length
$\omega$-circle.  Undoing the length normalisation gives a unique relaxed unnormalised flow converging, modulo
Euclidean motions and reparametrisation, to an $\omega$-circle of finite positive length.
\end{corollary}

\begin{proof}
After translating and rescaling the initial curve, Theorem~\ref{thm:canonical-W22-global} gives the unique global
unit-length barycentred relaxed flow.  Undoing the length normalisation gives the corresponding unnormalised relaxed
flow.  Since the length-normalised solution converges smoothly to the unit-length \(\omega\)-circle and the dilation
coefficient is integrable after entrance into the small-\(E_m\) basin, the unnormalised length converges to a finite
positive limit.  Thus the unnormalised relaxed flow converges, modulo Euclidean motions and reparametrisation, to the
corresponding \(\omega\)-circle.
\end{proof}

\begin{proof}[Proof of Theorem~\ref{T:small-Kosc}]
The length-normalised statement is Theorem~\ref{thm:canonical-W22-global}.  The equivalent unnormalised
formulation follows from Corollary~\ref{cor:W22-relaxed-unnormalised} and the reconstruction formula
\eqref{eq:length-reconstruct-main}.
\end{proof}


\appendix

\section{Evolution identities for the barycentred length-normalised flow}
\label{A:normalised-evolution}

This appendix proves the evolution identities stated in Section~\ref{S:len-normalised}.

\begin{proof}[Proof of Proposition~\ref{prop:barycentred-gauge-identities}]
Let
\[
\partial_\tau\eta=F N+\beta_FT-b_F.
\]
The vector $-b_F$ decomposes as
\[
-b_F=-\langle b_F,N\rangle N-\langle b_F,T\rangle T.
\]
For a velocity $V_nN+V_tT$ one has
\[
\partial_\tau(ds)=(\partial_sV_t-kV_n)\,ds.
\]
The translation contribution cancels because
\[
\partial_s\bigl(-\langle b_F,T\rangle\bigr)
-
k\bigl(-\langle b_F,N\rangle\bigr)
=
-k\langle b_F,N\rangle+k\langle b_F,N\rangle
=0.
\]
Thus
\[
\partial_\tau(ds)=\bigl((\beta_F)_s-kF\bigr)\,ds=0
\]
by the gauge condition $(\beta_F)_s=kF$.  Hence the length and the constant-speed parametrisation are preserved.

Since $ds$ is fixed,
\[
\frac{\dd}{\dd \tau}\int_\eta\eta\,ds
=
\int_\eta\partial_\tau\eta\,ds
=
\int_\eta(FN+\beta_FT-b_F)\,ds
=
b_F-b_F=0.
\]
This proves the preservation of the barycentric gauge.

Next, length preservation gives
\[
0
=
\frac{\dd}{\dd \tau}L[\eta]
=
-\int_\eta kF\,ds
=
-\int_\eta k\K_m\,ds-\Lambda\int_\eta kh\,ds.
\]
The Minkowski identity follows from $q_s=1+kh$:
\[
0=\int_\eta q_s\,ds=1+\int_\eta kh\,ds,
\]
so $\int_\eta kh\,ds=-1$.  Therefore
\[
\Lambda=\int_\eta k\K_m\,ds.
\]

Now let $\gamma$ be an unnormalised smooth flow and define
\[
\eta=\rho(\gamma-\bar\gamma),
\qquad
\rho=L_\gamma^{-1},
\qquad
\frac{d\tau}{dt}=\rho^{2m+4}.
\]
The scaling laws are
\[
ds=\rho\,ds_\gamma,
\qquad
k=\rho^{-1}k^\gamma,
\qquad
\K_m=\rho^{-(2m+3)}\K_m^\gamma .
\]
Since
\[
\frac{\dd}{\dd t }L_\gamma=-\int_{\gamma(t)}k^\gamma\K_m^\gamma\,ds_\gamma,
\]
we have
\[
\frac{\rho'}{\rho}
=
-\frac{L_\gamma'}{L_\gamma}
=
\rho\int_{\gamma(t)}k^\gamma\K_m^\gamma\,ds_\gamma.
\]
Therefore
\[
\rho^{-(2m+4)}\frac{\rho'}{\rho}
=
\rho^{-(2m+3)}
\int_{\gamma(t)}k^\gamma\K_m^\gamma\,ds_\gamma
=
\int_{\eta(\tau)}k\K_m\,ds
=
\Lambda .
\]
This proves \eqref{eq:log-length-main}.  Differentiating
\[
\eta=\rho(\gamma-\bar\gamma)
\]
and rewriting in $\tau$-time gives the normal velocity
\[
(\K_m+\Lambda h)N
\]
and tangential velocity equal to the constant-speed gauge up to an additive constant.  The additive constant only
generates a constant shift of the parameter on $\R/\Z$.  Subtracting the spatial average of the uncentred velocity is
exactly the term $-b_F$, and hence the centred representative solves \eqref{eq:NGFm-section}.  Integrating
\eqref{eq:log-length-main} gives \eqref{eq:length-reconstruct-main}.
\end{proof}

\begin{lemma}[Translation terms are intrinsically invisible]\label{lem:appendix-translation-invisible}
Let $b\in\R^2$ be independent of $s$.  Decompose the velocity $-b$ as
\[
-b=F_bN+\alpha_bT,
\qquad
F_b=-\langle b,N\rangle,
\qquad
\alpha_b=-\langle b,T\rangle .
\]
Then
\[
(\alpha_b)_s-kF_b=0,
\qquad
(F_b)_{ss}+k^2F_b+\alpha_bk_s=0.
\]
\end{lemma}

\begin{proof}
Since $b$ is constant,
\[
(\alpha_b)_s
=
-\partial_s\langle b,T\rangle
=
-\langle b,kN\rangle
=
kF_b.
\]
Also
\[
(F_b)_s
=
-\partial_s\langle b,N\rangle
=
\langle b,kT\rangle,
\]
and hence
\[
(F_b)_{ss}
=
k_s\langle b,T\rangle+k^2\langle b,N\rangle.
\]
Thus
\[
(F_b)_{ss}+k^2F_b+\alpha_bk_s
=
k_s\langle b,T\rangle+k^2\langle b,N\rangle
-k^2\langle b,N\rangle
-k_s\langle b,T\rangle
=0.
\]
\end{proof}

\begin{proof}[Proof of Proposition~\ref{prop:normalised-evolution-identities}]
We compute intrinsic quantities using the translation-reduced velocity
\[
F N+\beta_FT,
\qquad
F=\K_m+\Lambda h,
\qquad
(\beta_F)_s=kF.
\]
This is legitimate by Lemma~\ref{lem:appendix-translation-invisible}.

For a scalar quantity $u$ along the curve,
\[
\frac{\dd}{\dd \tau}\int_\eta u\,ds
=
\int_\eta \partial_\tau u\,ds
=
\int_\eta(\nabla_\tau u+\beta_Fu_s)\,ds
=
\int_\eta(\nabla_\tau u-kFu)\,ds,
\]
because $(\beta_F)_s=kF$.  The curvature evolution for the velocity $F N+\beta_FT$ is
\[
\partial_\tau k=F_{ss}+k^2F+\beta_Fk_s,
\]
hence
\[
\nabla_\tau k=F_{ss}+k^2F.
\]
Substituting $F=\K_m+\Lambda h$ and using
\[
h_{ss}+k^2h=-k-qk_s
\]
gives
\[
\nabla_\tau k
=
(\K_m)_{ss}+k^2\K_m-\Lambda k-\Lambda qk_s,
\]
which is \eqref{eq:k-evol-main}.

Since $\partial_\tau(ds)=0$ in the fixed gauge, $[\partial_\tau,D]=0$.  Therefore
\[
[\nabla_\tau,D]=[\partial_\tau-\beta_FD,D]=(\beta_F)_sD=kFD.
\]
By induction,
\begin{equation}\label{eq:appendix-commutator-power}
[\nabla_\tau,D^p]
=
\sum_{j=0}^{p-1}\binom{p}{j+1}D^j(kF)\,D^{p-j}
\end{equation}
for $p\ge1$, with empty sum for $p=0$.

The structure of $\K_m$ follows from the first variation formula \eqref{E:Km}:
\[
\K_m
=
(-1)^{m+1}D^{2m+2}k+P_3^{2m}(k),
\]
and differentiating gives \eqref{eq:Km-derivative-main}.

Let
\[
u:=D^pk.
\]
Using the integral identity and \eqref{eq:appendix-commutator-power},
\begin{align}
\label{eq:appendix-Hp-start}
\frac{\dd}{\dd \tau}H_p
&=
2\int_\eta u\,\nabla_\tau u\,ds
-
\int_\eta kF\,u^2\,ds \nonumber\\
&=
2\int_\eta u\,D^p(F_{ss}+k^2F)\,ds
+
2\sum_{j=0}^{p-1}\binom{p}{j+1}
\int_\eta u\,D^j(kF)\,D^{p-j}k\,ds
-
\int_\eta kF\,u^2\,ds .
\end{align}

We first isolate from \eqref{eq:appendix-Hp-start} all terms obtained by taking the $\Lambda h$ component of the
normal speed $F=\K_m+\Lambda h$; the remaining terms are those containing $\K_m$.  This is the contribution generated
by the infinitesimal rescaling in the length-normalised flow, and we call it the dilation contribution.  This is only
 bookkeeping. 
 When differentiating $D^pk$, one must also include the corresponding $\Lambda h$ contribution to
the commutator \eqref{eq:appendix-commutator-power}.  Since
\[
h_{ss}+k^2h=-k-qk_s,
\]
the direct $\Lambda h$ contribution to the curvature equation is
\[
\left.\nabla_\tau k\right|_\Lambda
=
\Lambda(h_{ss}+k^2h)
=
-\Lambda(k+qk_s).
\]
Applying $D^p$ to this direct term and adding the $\Lambda h$ part of the commutator gives
\[
\left.\nabla_\tau u\right|_\Lambda
=
-\Lambda u-\Lambda D^p(qk_s)
+
\Lambda
\sum_{j=0}^{p-1}\binom{p}{j+1}D^j(kh)\,D^{p-j}k .
\]
Consequently,
\begin{align}
\label{eq:appendix-Lambda-raw}
\left.\frac{\dd}{\dd \tau}H_p\right|_\Lambda
&=
\Lambda\Biggl[
-2H_p
-2\int_\eta uD^p(qk_s)\,ds \nonumber\\
&\qquad
+
2\sum_{j=0}^{p-1}\binom{p}{j+1}
\int_\eta uD^j(kh)D^{p-j}k\,ds
-
\int_\eta kh\,u^2\,ds
\Biggr].
\end{align}
We now simplify the bracket.  For $j\ge1$, $D^j(kh)=D^{j+1}q$, while $kh=q_s-1$ for $j=0$.  Hence
\[
\sum_{j=0}^{p-1}\binom{p}{j+1}
\int_\eta uD^j(kh)D^{p-j}k\,ds
=
\sum_{r=1}^{p}\binom{p}{r}
\int_\eta uD^rq\,D^{p-r+1}k\,ds
-
pH_p.
\]
Also,
\[
D^p(qk_s)
=
q\,u_s+
\sum_{r=1}^{p}\binom{p}{r}D^rq\,D^{p-r+1}k.
\]
Therefore
\[
-2\int_\eta uD^p(qk_s)\,ds
+
2\sum_{j=0}^{p-1}\binom{p}{j+1}
\int_\eta uD^j(kh)D^{p-j}k\,ds
=
-2\int_\eta q\,u\,u_s\,ds-2pH_p .
\]
Since
\[
0=\int_\eta(q\,u^2)_s\,ds
=
\int_\eta q_su^2\,ds+2\int_\eta q\,u\,u_s\,ds
\]
and $q_s=1+kh$, we obtain
\[
-2\int_\eta q\,u\,u_s\,ds
=
H_p+\int_\eta kh\,u^2\,ds.
\]
Substituting this into \eqref{eq:appendix-Lambda-raw} gives the homogeneity law
\begin{equation}\label{eq:appendix-Lambda-Hp}
\left.\frac{\dd}{\dd \tau}H_p\right|_\Lambda
=
-(2p+1)\Lambda H_p.
\end{equation}

It remains to treat $F=\K_m$.  The leading term in
\[
2\int_\eta u\,D^p(\K_m)_{ss}\,ds
\]
is
\[
2(-1)^{m+1}\int_\eta D^pk\,D^{p+2m+4}k\,ds
=
-2\int_\eta D^{p+m+2}k\,D^{p+m+2}k\,ds,
\]
after integrating by parts $m+2$ times.  The remaining terms from $D^p(\K_m)_{ss}$ contribute
\[
2\int_\eta k_{s^p}P_3^{p+2m+2}(k)\,ds.
\]
The terms from $D^p(k^2\K_m)$ and from the commutator sum in
\eqref{eq:appendix-Hp-start} have the Dziuk--Kuwert--Sch\"atzle structure
\[
2\int_\eta k_{s^p}P_3^{p+2m+2}(k)\,ds
+
2\int_\eta k_{s^p}P_5^{p+2m}(k)\,ds,
\]
after enlarging the polynomial symbols.  The metric term with $F=\K_m$ is
\[
-\int_\eta k\K_m k_{s^p}^2\,ds.
\]
Together with \eqref{eq:appendix-Lambda-Hp}, this proves \eqref{eq:Hp-evol-main}.

For the energy, the first variation gives
\[
\frac{\dd}{\dd \tau}E_m[\eta(\tau)]
=
-\int_\eta \K_mF\,ds
=
-\int_\eta\K_m^2\,ds-\Lambda\int_\eta h\K_m\,ds .
\]
The dilation identity
\[
\int_\eta h\K_m\,ds=(2m+1)E_m
\]
follows by differentiating the scaling law
\[
E_m[\rho\eta]=\rho^{-(2m+1)}E_m[\eta]
\]
at $\rho=1$ and using the first variation with dilation vector field $\eta=qT+hN$.  This proves
\eqref{eq:Em-evol-main}.

Finally, insert \eqref{E:Km} into
\[
\Lambda=\int_\eta k\K_m\,ds.
\]
The leading term gives
\[
\int_\eta (-1)^{m+1}kD^{2m+2}k\,ds
=
\int_\eta k_{s^{m+1}}^2\,ds,
\]
and all remaining terms are quartic with total derivative order $2m$.  This proves
\eqref{eq:Lambda-structure-main}.
\end{proof}

\section{A priori Sobolev estimates for the barycentred length-normalised flow}
\label{A:normalised-apriori}

This appendix gives the interpolation estimates behind Proposition~\ref{prop:normalised-Sobolev-control} and the
positive-time smoothing statement used in the small-\(\Kosc\) basin.

\begin{lemma}[Periodic Gagliardo--Nirenberg]\label{lem:appendix-GN}
Let $R\in\N$, $0\le j<R$, and $q\in[2,\infty]$.  There is a constant $C=C(R,j,q)$ such that every smooth periodic
function $u$ on a unit-length circle satisfies
\begin{equation}\label{eq:appendix-GN}
\|D^ju\|_{L^q}
\le
C\|D^Ru\|_{L^2}^{\theta}\|u\|_{L^2}^{1-\theta}
+
C\|u\|_{L^2},
\end{equation}
where
\[
j+\frac12-\frac1q
=
\theta\left(R-\frac12\right).
\]
\end{lemma}

\begin{proof}
This is the standard one-dimensional periodic Gagliardo--Nirenberg inequality on an interval of length one.  The
additional lower-order term accounts for the mean of $u$.
\end{proof}

\begin{lemma}[DKS absorption estimate]\label{lem:appendix-DKS-absorption}
Fix $R\in\N$.  Let
\[
M(k)=\int_\eta \prod_{i=1}^a |D^{j_i}k|\,ds
\]
be a monomial integral with $a\ge2$ and
\[
j_1+\cdots+j_a\le 2R-2.
\]
Then, for every $\delta>0$,
\begin{equation}\label{eq:appendix-DKS-absorption}
M(k)
\le
\delta\|D^Rk\|_2^2
+
C_{\delta}\mathcal P(1,H_0,\dots,H_{R-1}),
\end{equation}
where $\mathcal P$ is a polynomial with nonnegative coefficients depending only on the monomial, and
$C_\delta$ depends only on $\delta$, the monomial, $R$, and the fixed turning number.
\end{lemma}

\begin{proof}
If one of the factors has derivative order at least $R$, integrate by parts until all factors have derivative order at
most $R-1$, leaving only monomial integrals with total derivative order at most $2R-2$.  Apply H\"older's
inequality, placing one factor in $L^2$ and the remaining factors in suitable $L^{q_i}$ spaces.  Estimate each
factor using Lemma~\ref{lem:appendix-GN}.  Terms involving only lower-order $L^2$ norms are included directly in
$\mathcal P(1,H_0,\dots,H_{R-1})$.  For the remaining terms, the total derivative order assumption gives a power of
$\|D^Rk\|_2$ strictly less than $2$, so Young's inequality gives
\[
M(k)
\le
\delta\|D^Rk\|_2^2
+
C_\delta\mathcal P(1,H_0,\dots,H_{R-1}).
\]
The fixed mean curvature
\[
\bar k=2\pi\omega
\]
is absorbed into the polynomial through the constant term.
\end{proof}

\begin{proof}[Proof of Proposition~\ref{prop:normalised-Sobolev-control}]
Set
\[
R:=p+m+2.
\]
In \eqref{eq:Hp-evol-main}, the leading term is
\[
-2\|D^Rk\|_2^2.
\]
Every monomial in
\[
\int_\eta k_{s^p}P_3^{p+2m+2}(k)\,ds
\]
has at least four factors and total derivative order
\[
p+(p+2m+2)=2R-2.
\]
Every monomial in
\[
\int_\eta k_{s^p}P_5^{p+2m}(k)\,ds
\]
has at least six factors and total derivative order
\[
p+(p+2m)=2R-4.
\]
For the metric term, use
\[
\K_m=(-1)^{m+1}D^{2m+2}k+P_3^{2m}(k).
\]
The top-order part of
\[
\int_\eta k\K_mk_{s^p}^2\,ds
\]
has total derivative order
\[
2m+2+2p=2R-2,
\]
and the lower-order part has still smaller total derivative order.  Hence Lemma~\ref{lem:appendix-DKS-absorption}
applies to all these terms.

It remains to treat $\Lambda H_p$.  By \eqref{eq:Lambda-structure-main},
\[
|\Lambda|
\le
\|D^{m+1}k\|_2^2
+
\left|\int_\eta P_4^{2m}(k)\,ds\right|.
\]
The first term is controlled by interpolation between $H_0$ and $\|D^Rk\|_2^2$, while the quartic integral is
controlled by Lemma~\ref{lem:appendix-DKS-absorption}.  Multiplying by $H_p$ and applying Young's inequality gives
\[
|\Lambda|H_p
\le
\delta\|D^Rk\|_2^2
+
C_\delta\mathcal P_{m,p}(1,H_0,\dots,H_{R-1}).
\]
Substituting these estimates into \eqref{eq:Hp-evol-main} proves \eqref{eq:Sobolev-diff-main}.
\end{proof}

\begin{proof}[Proof of Corollary~\ref{cor:positive-time-smoothing-Kosc}]
For smooth data, Proposition~\ref{prop:Kosc-decay} gives
\[
e(\tau)\le C e(0)e^{-c\tau}
\]
and
\[
\int_\tau^{\tau+1}\mathfrak D_m[f(\sigma)]\,d\sigma
\le
C e(\tau).
\]
The spectral estimate \eqref{eq:Dm-controls-derivatives-main} gives spacetime control of
\[
\|D^rf\|_2^2
\qquad(0\le r\le m+2)
\]
on each interval $[\tau,\tau+1]$.  Since $k=\kappa+f$, the same is true for the corresponding curvature norms.

For higher derivatives, apply the differential inequality \eqref{eq:Sobolev-diff-main} successively.  On each unit
time interval in the maximal existence interval, the lower-order terms are already controlled from the previous
step, while the leading term gives spacetime control of the next block of derivatives.  The uniform Gronwall lemma
then yields pointwise bounds for $H_p$ on every interval $[\delta,T)$, with constants independent of $T$, for every
$p$.  These bounds rule out finite-time singularity by the continuation criterion, so the maximal interval is
$[0,\infty)$.  The $L^\infty$ bounds follow from Lemma~\ref{lem:appendix-GN}.

Equivalently, one may express this argument in the curvature-coordinate formulation of Section~\ref{S:existence}:
small \(\Kosc\) keeps the \(L^2\)-curvature coordinate inside the same small ball, so the analytic-semigroup local
theory can be restarted uniformly, and the semigroup smoothing estimates give the displayed bounds.

For relaxed \(W^{2,2}\) solutions, approximate the initial curve by smooth data using
Lemma~\ref{lem:W22-density-closure}.  The smooth estimates above are uniform for approximating data with the same
small oscillation bound.  Passing to the limit by the continuous-dependence statement in Theorem~\ref{thm:W22-LWP}
gives the result for the canonical relaxed flow.
\end{proof}

\section{Linearisation about an \texorpdfstring{$\omega$}{omega}-circle and the spectral gap}\label{SS:linear-circle}

Let $\eta_\circ$ be an $\omega$-circle of length $1$. Then its curvature is constant,
\[
k_\circ \equiv 2\pi\,\omega,
\qquad
E_m[\eta_\circ]=0,
\qquad
\K_m[\eta_\circ]\equiv 0.
\]

\begin{lemma}\label{lem:linearised-graph-definite}
Let $\eta$ be a normal graph over $\eta_\circ$ with normal displacement $\varepsilon\phi$:
\[
\eta=\eta_\circ+\varepsilon\phi\,N_\circ+O(\varepsilon^2),
\qquad \phi:\R/\Z\to\R.
\]
Then the curvature variation, with the normal orientation used throughout the paper, is
\begin{equation}\label{eq:delta-k-circle}
f=\delta k = \phi_{ss}+k_\circ^{\,2}\phi,
\end{equation}
and the linearisation of $\K_m$ can be written as
\begin{equation}\label{eq:Km-linear-phi}
\K_m
=
(-1)^{m+1}\,\varepsilon\,D^{2m}\,(D^2+k_\circ^{\,2})^2\phi
+O(\varepsilon^2),
\qquad D:=\partial_s.
\end{equation}
In particular, the quadratic dissipation involves the nonnegative quadratic form
\[
\int_{\eta_\circ}\bigl[D^{2m}(D^2+k_\circ^{\,2})^2\phi\bigr]^2\,ds.
\]
Its kernel contains the infinitesimal translation modes $n=\pm\omega$; these modes also have zero quadratic
$E_m$-energy and are handled in the nonlinear argument by the closure condition.
\end{lemma}

\begin{proof}
The curvature variation formula consistent with \eqref{eq:k-var} is
$\delta k=\phi_{ss}+k^2\phi$; since $k\equiv k_\circ$ on $\eta_\circ$ this gives
\eqref{eq:delta-k-circle}.  Substituting $f=(D^2+k_\circ^{\,2})\phi$ into
\begin{equation}\label{eq:Km-linear-correct}
\K_m[k]
=
(-1)^{m+1}\varepsilon\Bigl(f_{s^{2m+2}}+k_\circ^{\,2}f_{s^{2m}}\Bigr)
+O(\varepsilon^2)
\qquad\text{in }L^2(ds)
\end{equation}
yields
\[
\K_m
=
(-1)^{m+1}\varepsilon\,D^{2m}(D^2+k_\circ^{\,2})\,f
+O(\varepsilon^2)
=
(-1)^{m+1}\varepsilon\,D^{2m}(D^2+k_\circ^{\,2})^2\phi
+O(\varepsilon^2),
\]
which is \eqref{eq:Km-linear-phi}.
\end{proof}

\begin{lemma}[Spectral gap for the linearised dissipation on unit length]\label{lem:spectral-gap-linearised}
Let $\phi:\R/\Z\to\R$ be smooth with $\int_0^1 \phi\,ds=0$ (this is the linearised length constraint on
$\eta_\circ$).  Set
\begin{equation}\label{eq:linearised-gap-constant}
\nu_{m,\omega}
:=
(2\pi)^{2m+4}
\min_{n\in\Z\setminus\{0,\pm\omega\}}
|n|^{2m}(n^2-\omega^2)^2
>0 .
\end{equation}
Then
\begin{equation}\label{eq:spectral-gap-linearised}
\int_0^1\bigl[D^{2m}(D^2+k_\circ^{\,2})^2\phi\bigr]^2\,ds
\ge
\nu_{m,\omega}
\int_0^1\bigl[D^{m}(D^2+k_\circ^{\,2})\phi\bigr]^2\,ds.
\end{equation}
Equivalently, in terms of the quadratic approximation
\[
E_m^{(2)}(\phi):=\frac12\int_0^1\bigl[D^{m}(D^2+k_\circ^{\,2})\phi\bigr]^2\,ds
\quad\Bigl(\;=\frac12\int_0^1 f_{s^m}^2\,ds\ \text{with }f=(D^2+k_\circ^{\,2})\phi\Bigr),
\]
we have
\begin{equation}\label{eq:spectral-gap-linearised-Em}
\int_0^1\bigl[D^{2m}(D^2+k_\circ^{\,2})^2\phi\bigr]^2\,ds
\ge
2\nu_{m,\omega}E_m^{(2)}(\phi).
\end{equation}
\end{lemma}

\begin{proof}
Expand $\phi(s)=\sum_{n\in\Z}\hat\phi_n e^{2\pi i ns}$ on $\R/\Z$.  The mean-zero condition gives
$\hat\phi_0=0$.  For $n\neq0$,
\[
D^{2m}(D^2+k_\circ^{\,2})^2 e^{2\pi i ns}
=
(2\pi i n)^{2m}\bigl(k_\circ^{\,2}-(2\pi n)^2\bigr)^2 e^{2\pi i ns},
\]
and
\[
D^{m}(D^2+k_\circ^{\,2}) e^{2\pi i ns}
=
(2\pi i n)^{m}\bigl(k_\circ^{\,2}-(2\pi n)^2\bigr)e^{2\pi i ns}.
\]
Since $\kappa=k_\circ=2\pi\omega$, Parseval's identity gives
\begin{align*}
\int_0^1\bigl[D^{2m}(D^2+k_\circ^{\,2})^2\phi\bigr]^2ds
&=
\sum_{n\neq0}
(2\pi)^{4m+8}|n|^{4m}(n^2-\omega^2)^4|\hat\phi_n|^2,\\
\int_0^1\bigl[D^{m}(D^2+k_\circ^{\,2})\phi\bigr]^2ds
&=
\sum_{n\neq0}
(2\pi)^{2m+4}|n|^{2m}(n^2-\omega^2)^2|\hat\phi_n|^2.
\end{align*}
The modes $n=\pm\omega$ contribute zero to both sums.  For every remaining nonzero mode,
\[
(2\pi)^{2m+4}|n|^{2m}(n^2-\omega^2)^2
\ge
\nu_{m,\omega},
\]
by the definition of $\nu_{m,\omega}$.  Multiplying this modewise inequality by the corresponding coefficient in
the second sum and summing proves \eqref{eq:spectral-gap-linearised}.  The reformulation
\eqref{eq:spectral-gap-linearised-Em} is immediate from the definition of $E_m^{(2)}$.
\end{proof}


\begin{bibdiv}
\begin{biblist}


\bib{A95}{book}{
  author={Amann, Herbert},
  title={Linear and quasilinear parabolic problems. Vol. I},
  subtitle={Abstract linear theory},
  series={Monographs in Mathematics},
  volume={89},
  publisher={Birkh\"auser},
  date={1995},
}

\bib{A05}{article}{
  author={Amann, Herbert},
  title={Quasilinear parabolic problems via maximal regularity},
  journal={Adv. Differential Equations},
  volume={10},
  number={10},
  date={2005},
  pages={1081--1110},
}


\bib{AMWW20}{article}{
  author={Andrews, Ben},
  author={McCoy, James},
  author={Wheeler, Glen},
  author={Wheeler, Valentina-Mira},
  title={Closed ideal planar curves},
  journal={Geom. Topol.},
  volume={24},
  number={2},
  date={2020},
  pages={1019--1049},
}

\bib{AW25}{article}{
  author        = {Andrews, Ben},
  author={Wheeler, Glen},
  title         = {On the planar free elastic flow with small oscillation of curvature},
  date          = {2025},
    status={preprint},
  eprint        = {arXiv:2509.11129},
}


\bib{AW26}{article}{
  author={Andrews, Ben},
  author={Wheeler, Glen},
  title={Jellyfish exist},
  status={preprint},
  eprint={arXiv:2601.21227},
  date={2026},
}

\bib{B}{book}{
  author={Baker, Charles},
  title={The mean curvature flow of submanifolds of high codimension},
  date={2011},
  publisher={PhD Thesis, Australian National University; also arXiv:1104.4409},
}

\bib{DKS02}{article}{
  author={Dziuk, Gerhard},
  author={Kuwert, Ernst},
  author={Sch\"atzle, Reiner},
  title={Evolution of elastic curves in {$\mathbb{R}^n$}: existence and computation},
  journal={SIAM J. Math. Anal.},
  volume={33},
  number={5},
  pages={1228--1245},
  date={2002},
}

\bib{HT}{inproceedings}{
  author={Harary, Gur},
  author={Tal, Ayellet},
  title={3D Euler spirals for 3D curve completion},
  booktitle={Proceedings of the Twenty-Sixth Annual Symposium on Computational Geometry},
  date={2010},
  pages={393--402},
  organization={ACM},
}

\bib{H81}{book}{
  author={Henry, Daniel},
  title={Geometric theory of semilinear parabolic equations},
  series={Lecture Notes in Mathematics},
  volume={840},
  publisher={Springer-Verlag},
  date={1981},
}
\bib{Levien2008EulerSpiralHistory}{article}{
  author      = {Levien, Raphael Linus},
  title       = {The Euler Spiral: A Mathematical History},
  institution = {EECS Department, University of California, Berkeley},
  number      = {UCB/EECS-2008-111},
  year        = {2008},
  url         = {https://www2.eecs.berkeley.edu/Pubs/TechRpts/2008/EECS-2008-111.html},
}

\bib{Levien2009FromSpiralToSpline}{article}{
  author = {Levien, Raphael Linus},
  title  = {From Spiral to Spline: Optimal Techniques in Interactive Curve Design},
  school = {EECS Department, University of California, Berkeley},
  year   = {2009},
  number = {UCB/EECS-2009-162},
  url    = {https://www2.eecs.berkeley.edu/Pubs/TechRpts/2009/EECS-2009-162.html},
}



\bib{LX}{article}{
  author={Liu, D},
  author={Xu, G},
  title={A general sixth order geometric partial differential equation and its application in surface modeling},
  journal={J. Inf. Comput. Sci.},
  volume={4},
  date={2007},
  pages={1--12},
}

\bib{L95}{book}{
  author={Lunardi, Alessandra},
  title={Analytic semigroups and optimal regularity in parabolic problems},
  series={Progress in Nonlinear Differential Equations and their Applications},
  volume={16},
  publisher={Birkh\"auser},
  date={1995},
  label={L95},
}


\bib{MWW1}{article}{
  author={McCoy, James},
  author={Wheeler, Glen},
  author={Wu, Yuhan},
  title={A sixth order flow of plane curves with boundary conditions},
  journal={Tohoku Math. J.},
  volume={72},
  number={3},
  date={2020},
  pages={379--393},
}

\bib{MWW2}{article}{
  author={McCoy, James},
  author={Wheeler, Glen},
  author={Wu, Yuhan},
  title = {High order curvature flows of plane curves with generalised Neumann boundary conditions},
pages = {497--513},
volume = {15},
number = {3},
journal = {Advances in Calculus of Variations},
year = {2022},
}


\bib{MWW}{article}{
  author  = {McCoy, James},
  author={Wheeler, Glen},
  author={Wu, Yuhan},
  title   = {A Length-Constrained Ideal Curve Flow},
  journal = {The Quarterly Journal of Mathematics},
  volume  = {73},
  number  = {2},
  pages   = {685--699},
  year    = {2022},
  doi     = {10.1093/qmath/haab050},
  note    = {Published online 15 November 2021},
}


\bib{MW25}{article}{
  author  = {Miura, Tatsuya},
  author={Wheeler, Glen},
  title   = {The free elastic flow for closed planar curves},
  journal = {Journal of Functional Analysis},
  volume  = {289},
  number  = {7},
  pages   = {111030},
  year    = {2025},
  doi     = {10.1016/j.jfa.2025.111030},
}

\bib{OW261}{article}{
  author = {Okabe, Shinya},
  author = {Wheeler, Glen},
  title = {On ideal lemniscates and circular waves},
  journal = {preprint},
}

\bib{OW262}{article}{
  author = {Okabe, Shinya},
  author = {Wheeler, Glen},
  title = {Finite-time blowup and asymptotic circularity for the length-penalised ideal flow of planar curves},
  journal = {preprint},
}


\bib{PW16}{article}{
  author={Parkins, Scott},
  author={Wheeler, Glen},
  title={The polyharmonic heat flow of closed plane curves},
  journal={J. Math. Anal. Appl.},
  volume={439},
  date={2016},
  pages={608--633},
}

\bib{P03}{article}{
  author={Poppenberg, Martin},
  title={Nash-Moser techniques for nonlinear boundary-value problems},
  journal={Electron. J. Differential Equations},
  volume={2003},
  number={54},
  date={2003},
  pages={1--33},
}

\bib{UW}{article}{
  author={Ugail, Hassan},
  author={Wilson, Michael},
  title={Modeling of oedemous limbs and venous ulcers using partial differential equations},
  journal={Theor. Biol. Med. Model.},
  volume={2},
  number={28},
  date={2005},
}






\bib{Wu2021ShortTimeExistence}{article}{
  author    = {Wu, Yuhan},
  title     = {Short time existence for higher order curvature flows with and without boundary conditions},
  booktitle = {2019--20 {MATRIX} Annals},
  series    = {MATRIX Book Series},
  volume    = {4},
  publisher = {Springer},
  year      = {2021},
  pages     = {773--783},
}


\end{biblist}
\end{bibdiv}
\end{document}